\documentclass[10pt,a4paper,reqno]{amsart}
\usepackage[utf8]{inputenc}
\usepackage[T1]{fontenc}       
\usepackage{ae}   
\usepackage{amsfonts}         
\usepackage{amsmath}
\usepackage{amsthm}
\usepackage{amssymb}
\usepackage{mathtools}
\usepackage{calc}
\usepackage{centernot}
\usepackage{color}
\usepackage[pdftex]{hyperref}
\usepackage{mathrsfs}

\usepackage{bigints} 

\usepackage{eucal}
\usepackage{calligra}
\DeclareMathAlphabet{\mathpzc}{OT1}{pzc}{m}{it}
\usepackage{bbold}

\usepackage{xargs}

\usepackage{todonotes}

\newcommandx{\huom}[2][1=]{\todo[linecolor=red,backgroundcolor=red!10,bordercolor=red,#1]{#2}}

\newcommand{\ve}{\varepsilon}
\newcommand{\N}{\mathbb{N}}

\newcommand{\R}{\mathbb{R}}

\newcommand{\HH}{\mathbb{H}}

\newcommand{\ang}[1]{\left\langle #1 \right\rangle}

\newcommand{\p}{\partial}

\DeclareMathOperator\diam{diam}
\DeclareMathOperator\dist{dist}
\DeclareMathOperator\dv{div}

\newcommand{\pinf}{\partial_{\infty}}

\DeclareMathOperator\Ric{Ric}

\DeclareMathOperator\Hess{Hess}
\DeclareMathOperator\vol{vol}

\DeclareMathOperator\supp{supp}
\DeclareMathOperator\sgn{sgn}

\numberwithin{equation}{section}

\theoremstyle{plain}
\newtheorem{thm}{Theorem}[section]
\newtheorem{lem}[thm]{Lemma}
\newtheorem{cor}[thm]{Corollary}
\newtheorem{prop}[thm]{Proposition}

\theoremstyle{definition}

\newtheorem{exa}[thm]{Example}
\newtheorem{rem}[thm]{Remark}

\begin{document}

\title{Asymptotic Dirichlet problems in warped products}

\author{Jean-Baptiste Casteras}
\address{J.-B. Casteras, Departement de Mathematique
Universite libre de Bruxelles, CP 214, Boulevard du Triomphe, B-1050 Bruxelles, Belgium}
\email{jeanbaptiste.casteras@gmail.com}

\author{Esko Heinonen}
\address{E. Heinonen, Department of Mathematics and Statistics, P.O.B. 68 (Pietari Kalmin katu 5), 00014 University
of Helsinki, Finland.}
\email{esko.heinonen@helsinki.fi}

\author{Ilkka Holopainen}
\address{I.Holopainen, Department of Mathematics and Statistics, P.O.B. 68 (Pietari Kalmin katu 5), 00014 University
of Helsinki, Finland.}
\email{ilkka.holopainen@helsinki.fi}

\author{Jorge Lira}
\address{J. Lira, Departamento de Matemática, Universidade Federal do Ceará, Bloco 914,
Campus do Pici, Fortaleza, Ceará 60455-760, Brazil}
\email{jorge.lira@mat.ufc.br}

\thanks{J.-B.C. supported by MIS F.4508.14 (FNRS)} 
\thanks{E.H. supported by Jenny and Antti Wihuri Foundation and CNPq.}
\thanks{I.H. supported by the Faculty of Science, University of Helsinki, and FUNCAP}
\thanks{J.L. supported by CNPq and FUNCAP}


\begin{abstract}
We study the asymptotic Dirichlet problem for Killing graphs with prescribed mean curvature $H$ in warped product manifolds $M\times_\varrho \R$. In the first part of the paper, we prove the existence of Killing graphs with prescribed boundary on geodesic balls under suitable assumptions on $H$ and the mean curvature of the Killing cylinders over geodesic spheres. In the process we obtain a uniform interior gradient estimate improving previous results by Dajczer and de Lira. In the second part we solve the asymptotic Dirichlet problem in a large class of manifolds whose sectional curvatures are allowed to go to $0$ or to $-\infty$ provided that $H$
satisfies certain bounds with respect to the sectional curvatures of $M$ and the norm of the Killing vector field. 
Finally we obtain non-existence results if the prescribed mean curvature function $H$ grows too fast.
\end{abstract}

\maketitle

\tableofcontents

\section{Introduction}

Let $N$ be a Riemannian manifold of the form $N = M \times_\varrho \mathbb{R},$ where $M$ is a complete $n$-dimensional Riemannian  manifold and $\varrho\in C^\infty(M)$ is a smooth (warping) function. This means that the Riemannian metric $\bar g$ in $N$ is of the form
\begin{equation}
\bar g =  (\varrho\circ\pi_1)^2\pi_2^*\,{\rm d} t^2 + \pi^*_1g,
\end{equation}
where $g$ denotes the Riemannian metric in $M$ whereas $t$ is the natural coordinate in $\mathbb{R}$ and $\pi_1: M\times \mathbb{R}\to M$ and $\pi_2:M\times \mathbb{R}\to \mathbb{R}$ are the standard projections. It follows that the coordinate vector field $X = \partial_t$ is a Killing field and that $\varrho = |X|$ on $M$. 
Since the norm of $X$ is preserved along its flow lines, we may extend $\varrho$ to a smooth function $\varrho=|X| \in C^\infty(N)$. From now on, we suppose that $\varrho>0$ on $M$. 

In this paper we study Killing graphs with prescribed mean curvature. Such graphs were introduced by Dajczer and Ripoll in \cite{DR}, where the Dirichlet problem for a graph of constant mean curvature $H$ with $C^{2,\alpha}$ boundary values was solved in a bounded domain $\Omega$ contained in a normal geodesic disk $D\subset M$ of radius $r_0$ under hypothesis involving $r_0$, data on $\Omega$, and the curvature of the ambient 3-dimensional space $N$. A bit later in \cite{DHL} the Dirichlet problem for prescribed mean curvature $H\in C^{\alpha}$ with $C^{2,\alpha}$ boundary values was solved in bounded domains $\Omega\subset M$ with $C^{2,\alpha}$ boundary again under hypothesis involving data on $\Omega$ and the Ricci curvature of the ambient space $N$. Recall that given a domain $\Omega \subset M$, the Killing graph of a $C^2$ function $u:\Omega\to \mathbb{R}$ is the hypersurface given by 
\[
\Sigma_u = \{(x, u(x)): x\in \Omega\}\subset M\times \mathbb{R}.
\]
In other words, 
\[
\Sigma_u=\{\Psi(x,u(x))\colon x\in\Omega\},
\]
where $\Psi\colon \Omega\times \R\to N$ is the flow generated by $X$.
In \cite{DLR} the Dirichlet problem was solved with merely continuous boundary data. Furthermore, the authors proved the existence and uniqueness of so-called radial graphs in the 
hyperbolic space $\HH^{n+1}$ with prescribed mean curvature and asymptotic boundary data at infinity thus solving the asymptotic Dirichlet problem in $\HH^n\times_{\cosh r}\R$. One of our goals in the current paper is to solve the asymptotic Dirichlet problem  with prescribed mean curvature in a large class of negatively curved manifolds.

On the other hand, it is an interesting question under which conditions on a Riemannian manifold $M$ every entire constant mean curvature graph over $M$ is a slice, i.e. a graph of a constant function. The first such result is the celebrated theorem due to Bombieri, De Giorgi, and Miranda \cite{bombieri} that an entire
minimal positive graph over $\R^n$ is a totally geodesic slice. 
Their result was extended by Rosenberg, Schulze, and Spruck \cite{RSS} to a complete Riemannian manifold $M$ with nonnegative Ricci curvature and the sectional curvature bounded from below by a negative constant. Ding, Jost, and Xin considered in \cite{DJX}
complete, noncompact Riemannian manifolds with nonnegative Ricci curvature, Euclidean volume growth, and quadratic decay of the curvature tensor. They proved that an entire minimal graph over such a manifold $M$ must be a slice if its height function has at most linear growth on one side unless $M$ is isometric to Euclidean space.
In the recent paper \cite{CHH3} Casteras, Heinonen, and Holopainen showed that a minimal positive graph over a complete Riemannian manifold with asymptotically nonnegative sectional curvature and only one end is a slice if its height function has at most linear growth.
Entire Killing graphs in $M\times_{\varrho}\R$ with constant mean curvature were studied in \cite{DL} and \cite{Dajczer2016}. In particular, it was shown in \cite{DL} that a bounded entire Killing graph of constant mean curvature must be a slice if $\Ric_M\ge 0, \ K_M\ge -K_0$ for some $K_0\ge 0$, and if $\varrho\ge\varrho_0>0$, with $||\varrho||_{C^2(M)}<\infty$.

Our current paper is inspired by the above mentioned research \cite{DHL}, \cite{DLR}, \cite{DL}, and \cite{Dajczer2016} on Killing graphs with prescribed mean curvature as well as by the recent paper
\cite{CHH2}. In the latter, the asymptotic Dirichlet problem for 
$f$-minimal graphs in Cartan-Hadamard manifolds $M$ has been studied. Recall that $f$-minimal hypersurfaces are natural generalizations of 
self-shrinkers which play a crucial role in the study of mean curvature flow. Moreover, they are minimal hypersurfaces of 
weighted manifolds  $M_f=\bigl(M,g,e^{-f}{\rm d}\vol_{M}\bigr)$, where $(M,g)$ is a complete Riemannian manifold with the 
Riemannian volume element ${\rm d}\vol_{M}$. 

Returning to the Killing graph $\Sigma_u$ of a function $u$, we note that the induced metric in $\Sigma_u$ has components
\begin{equation}
\label{metric}
g_{ij} + \varrho^2(x)u_i u_j,
\end{equation}
where $g_{ij}$ are local components of the metric $g$.  The induced volume element in $\Sigma_u$ (or equivalently, on the domain $\Omega \subset M$) is given by 
\[
 {\rm d}\Sigma = \varrho\sqrt{\varrho^{-2}+|\nabla u|^2}\, {\rm d}M.
\]
We consider the constrained area functional 
\[
\mathcal{A}_H[u] = \int_{\Omega}  \varrho\sqrt{\varrho^{-2}+|\nabla u|^2}\, {\rm d}M + \mathcal{V}_H[u],
\]
where
\[
\mathcal{V}_H[u] = \int_\Omega\int_0^{u(x)\varrho(x)} nH {\rm d}M = \int_\Omega nH \varrho u\, {\rm d}M
\]
and $H$ is a smooth function on $\Omega$.
Given an arbitrary compactly supported function $v\in C^\infty_0(\Omega)$ we have the first variation formula
\[
\delta \mathcal{A}_H[u]\cdot v = \frac{{\rm d}}{{\rm d}s}\Big|_{s=0}\mathcal{A}_H[u+sv]=
-\int_\Omega \Big({\rm div} \Big(\frac{\nabla u}{W}\Big) +\Big\langle \nabla \log\varrho, \frac{\nabla u}{W}\Big\rangle-nH\Big) v\varrho\, {\rm d}M,
\]
where
\[
W = \sqrt{\varrho^{-2}+|\nabla u|^2}
\]
and the differential operators $\nabla$ and ${\rm div}$ are taken with respect to the metric $g$ in $M$. Then the Euler-Lagrange equation of this functional is 
\begin{equation}
\label{PDE}
{\rm div} \Big(\frac{\nabla u}{W}\Big) +\Big\langle \nabla \log\varrho, \frac{\nabla u}{W}\Big\rangle=nH
\end{equation}
and $H(x)$ is the mean curvature of the graph $\Sigma_u\subset
M\times_\varrho\R$ at $(x,u(x))$. The equation \eqref{PDE} can be rewritten as 
\[
{\rm div}_{-\log\varrho} \Big(\frac{\nabla u}{W}\Big) =nH,
\]
where the weighted divergence operator corresponding to a smooth density function $f\in C^\infty(M)$ is defined by
\[
{\rm div}_f Z =  e^f{\rm div} (e^{-f} Z) =  {\rm div} Z -\langle \nabla f, Z\rangle.
\]
Note that this is the divergence-form operator that fits well with the weighted measure $\varrho \, {\rm d}M$ in the sense that a suitable version of the divergence theorem is still valid in this context. Reasoning another way around, since $\Sigma$ is oriented by the normal vector field
\[
N = \frac{1}{W}\big(\varrho^{-2} X - \nabla u |_{(x, u(x))}\big)
\]
and
\[
\Big\langle \nabla \log\varrho, \frac{\nabla u}{W}\Big\rangle = -\langle\bar\nabla \log\varrho, N\rangle,
\]
where $\bar\nabla$ is the Riemannian connection in $N$, we can interpret
\[
H_{\log\varrho} = H + \frac{1}{n}\langle \nabla \log\varrho, N\rangle 
\]
as a weighted mean curvature of the submanifold $\Sigma_u$ in the \emph{Riemannian} product $M\times \mathbb{R}$ in the sense that the Euler-Lagrange PDE may be rewritten as
\[
{\rm div} \Big(\frac{\nabla u}{W}\Big) =nH_{\log\varrho}.
\]
More generally, if $f$ is an arbitrary density in $M$ we consider a weighted area functional of the form
\[
\mathcal{A}_{H,f}[u] = \int_{\Omega}  e^{-f}\varrho\sqrt{\varrho^{-2}+|\nabla u|^2}\, {\rm d}M + \int_\Omega nH e^{-f} \varrho  u\, {\rm d}M.
\]
In this case, the Euler-Lagrange equation is
\begin{equation}
{\rm div}_{f} \Big(\frac{\nabla u}{W}\Big) +\Big\langle \nabla \log\varrho, \frac{\nabla u}{W}\Big\rangle=nH.
\end{equation}
As before, this equation may be rewritten either in terms of a modified weighted divergence
\[
{\rm div}_{f-\log\varrho} \Big(\frac{\nabla u}{W}\Big) =nH
\]
or as a prescribed weighted mean curvature equation
\begin{equation*}
{\rm div}_{f} \Big(\frac{\nabla u}{W}\Big):={\rm div} \Big(\frac{\nabla u}{W}\Big) +\langle \bar\nabla f, N\rangle =nH_{\log\varrho}.
\end{equation*}

For the time being, we restrict ourselves to the case where $f=0$. Intrinsically, given a hypersurface $\Sigma \subset N$ and denoting $u = t|_\Sigma$, the parametric counterpart of (\ref{PDE}) is 
\begin{equation}
\label{PDE-3}
\Delta_\Sigma u = nH\langle N, \partial_t\rangle - 2\langle \nabla^\Sigma\log\varrho, \nabla^\Sigma u\rangle,
\end{equation}
where $\Delta_\Sigma$ is the Laplace-Beltrami operator in $\Sigma$.  Indeed if $\nabla^\Sigma$ denotes the intrinsic covariant derivative in $\Sigma$, we have
\[
\nabla^\Sigma u = (\bar\nabla t)^T = \varrho^{-2}\partial_t^T,
\]
where $T$ denotes tangential projection onto $T\Sigma$. Hence we obtain
\[
\Delta_\Sigma u  = nH\varrho^{-2} \langle \partial_t, N\rangle + \langle \nabla^\Sigma \varrho^{-2}, \partial_t^T\rangle, 
\]
from where the formula (\ref{PDE-3}) above follows. 

In particular, minimal graphs in $N = M\times_\varrho\mathbb{R}$ have height function that satisfies the weighted harmonic equation
\begin{equation}
\label{harmonic}
\Delta_\Sigma u + 2\langle \nabla^\Sigma \log\varrho, \nabla^\Sigma u\rangle = 0.
\end{equation}
This may be considered as a PDE in $\Omega$ if we replace the metric $g$ by the induced metric with components given by (\ref{metric}).

Denoting
\[
\sigma^{ij}=g^{ij} - \frac{u^i u^j}{W^2} 
\]
we can write  (\ref{PDE}) in non-divergence form as
\begin{equation}
\label{non-div-0}
\sigma^{ij}u_{i;j} + (\log\varrho)^i u_i \left(1+\frac{1}{\varrho^2 W^2}\right)= nH W. 
\end{equation}

\section{Main results}
The existence of Killing graphs with prescribed mean curvature $H$ over bounded domains $\Omega\subset M$ with continuous boundary data on $\partial\Omega$ was established in \cite[Theorem 2]{DLR} under suitable conditions on the Ricci curvature on $\Omega$, the mean curvature function $H$, and on the mean curvature of the Killing cylinder over $\partial\Omega$; see also \cite{DHL}. 

In this paper we mainly focus on the setting where $M$ is a Cartan-Hadamard manifold with sectional curvatures controlled from above and below by some radial functions. We prove \emph{quantitative} a priori height and gradient estimates for solutions of \eqref{PDE} on geodesic balls $\Omega=B(o,k)\subset M$ under natural conditions on the prescribed mean curvature function in terms of
sectional curvatures $K_M$ and the warping function $\varrho$.
 These estimates allow us to use the continuity method (the Leray-Schauder method) and hence are enough to guarantee the existence of solutions to the following Dirichlet problem
\begin{equation}
\label{dir-aux}
\begin{cases}
{\rm div} \big(\frac{\nabla u}{W}\big) +\big\langle \nabla \log\varrho, \frac{\nabla u}{W}\big\rangle=nH & \quad \mbox{ in } \quad \Omega\\
u|\partial\Omega = \varphi & \quad \mbox{ in } \quad \partial\Omega,
\end{cases}
\end{equation} 
where $\varphi\in C(\partial\Omega)$.
We formulate the (local) existence result in geodesic balls on Cartan-Hadamard manifolds.
\begin{thm}\label{loc-exist}
Let $M$ be a Cartan-Hadamard manifold, 
$\Omega=B(o,k)\subset M$, and $\varphi\in C(\partial\Omega)$. Suppose that the prescribed mean curvature function 
$H\in C^\alpha(\Omega)$ satisfies
\[
|H(x)|<H_{k-d(x)}
\]
in $\bar{\Omega}$, where $d(x)=\dist\big(x,\partial B(o,k)\big)=k-r(x)$ and $H_{k-d}$ is the mean curvature of the Killing cylinder 
$\mathcal{C}_{k-d}$ over the geodesic sphere $\partial B(o, k-d)$.
Then there exists a unique solution 
$u\in C^{2,\alpha}(\Omega)\cap C(\bar{\Omega})$ to \eqref{dir-aux}.
\end{thm}
Above and in what follows we denote by $r(x)=d(x,o)$ the distance from $x$ to a fixed point $o\in M$.
We notice that the mean curvature of the Killing cylinder 
$\mathcal{C}_r$ over a geodesic sphere $\partial B(o,r)$ is given by
\[
H_r=\frac{1}{n}\left(\Delta r +\frac{1}{\varrho}\langle\nabla\varrho,\nabla r\rangle\right)
\]
and therefore can be estimated from below in terms of a suitable model manifold $M_{-a^2(r)}\times_{\varrho_+}\R$, where 
$M_{-a^2(r)}$ is a rotationally symmetric Cartan-Hadamard manifold with radial sectional curvatures equal to $-a^2(r)$ and 
$\varrho_+\colon M\to (0,\infty)$ is a positive rotationally symmetric  $C^1$ function such that
\begin{equation}\label{rho-hypo}
\frac{1}{\varrho}\langle\nabla\varrho,\nabla r\rangle=\frac{\partial_r\varrho}{\varrho}\ge\frac{\partial_r\varrho_+}{\varrho_+}.
\end{equation}
To formulate the next corollary and for later purposes we denote by 
$f_\kappa\in C^\infty([0,\infty))$ the solution
of the Jacobi equation
	\begin{equation}\label{jac-equ}
	\begin{cases}
	f_\kappa'' - \kappa^2 f_\kappa = 0 \\
	f_\kappa(0)=0 \\
	f_\kappa'(0) = 1,
	\end{cases}
	\end{equation}
whenever $\kappa\colon [0,\infty)\to[0,\infty)$ is a smooth function. 
\begin{cor}\label{loc-exist-cor}
Let $M$ be a Cartan-Hadamard manifold whose radial sectional curvatures are bounded from above by 
\[
K(P_x)\le -a\big(r(x)\big)^2
\]
for some smooth function $a\colon [0,\infty)\to[0,\infty)$. Suppose, moreover, that \eqref{rho-hypo} holds with some positive rotationally symmetric  $C^1$ function $\varrho_+=\varrho_+(r)$. If the prescribed mean curvature function 
$H\in C^\alpha(\Omega),\ \Omega=B(o,k),$ satisfies
\[
n|H(x)|<\frac{(n-1)f_a^\prime\big(r(x)\big)}{f_a\big(r(x)\big)}+
\frac{\varrho_+^\prime\big(r(x)\big)}{\varrho_+\big(r(x)\big)}
\]
for all $x\in\bar{\Omega}$, then there exists a unique solution 
$u\in C^{2,\alpha}(\Omega)\cap C(\bar{\Omega})$ to \eqref{dir-aux}.
\end{cor}
As mentioned above the proofs of Theorem~\ref{loc-exist} and Corollary~\ref{loc-exist-cor} for boundary data $\varphi\in C^{2,\alpha}(\partial\Omega)$ follow from the well-known continuity method once the a priori height and gradient estimates are at our disposal. The case of a continuous boundary values $\varphi\in C(\partial\Omega)$ can be treated as in \cite{DLR}; see also \cite{CHH2}.

Our main object in this paper is the asymptotic Dirichlet problem for Killing graphs with prescribed mean curvature and behaviour at infinity. To solve the problem, we extend the given  boundary value function
$\varphi\in C(\partial_\infty M)$ to a continuous function 
$\varphi\in C(\bar{M})$; see Section~\ref{sec_bar_infty} for the notation. Then we apply Corollary~\ref{loc-exist-cor} 
for an exhaustion $\Omega_k=B(o,k),\ k\in\N,$ of $M$ 
to obtain a sequence of 
solutions $u_k$ with boundary values $u_k|\partial\Omega_k=\varphi$. Under a suitable bound on $|H|$ in terms of a comparison manifold $M_{-a^2(r)}\times_{\varrho_+}\R$ we obtain a global height estimate and, consequently together with Schauder estimates, the sequence is uniformly bounded in the
$C^{2,\alpha}$-norm. Hence there exists a subsequence that converges in the $C^{2,\alpha}$-norm to a global solution $u$ to the equation
\[
{\rm div} \big(\frac{\nabla u}{W}\big) +\big\langle \nabla \log\varrho, \frac{\nabla u}{W}\big\rangle=nH
\]
in $M$. Finally, under suitable curvature upper and lower bounds as well as conditions on $|H|$ we are able to construct (local) barriers at infinity and prove that the solution $u$ extends continuously to 
$\partial_\infty M$ and attains the given boundary values $\varphi$ there.

The following two solvability theorems will be proven in 
Section~\ref{ADP_sec}.
\begin{thm}
\label{main1}
Let $M$ be a Cartan-Hadamard manifold satisfying the curvature
assumptions \eqref{curv-bound-gen} and \eqref{A1}--\eqref{A7} in Section \ref{sec_bar_infty}. Furthermore, assume that the prescribed mean curvature function 
$H\colon M \to \R$ satisfies the assumptions \eqref{height_mean_assum2} and 
\eqref{asym_meancurv_assum} with a convex warping function $\varrho$ satisfying \eqref{rho_height_assum}, \eqref{rho_height_assum1}, \eqref{rho_sign_assum}, and \eqref{rho_rad_assum}.
Then there exists a unique solution $u\colon M\to \R$
to the Dirichlet problem
	\begin{equation}\label{ADP1}
		\begin{cases}
	\dv_{-\log \varrho} \dfrac{\nabla u}{\sqrt{\varrho^{-2} + |\nabla u|^2}} = nH(x) \quad \text{in } M \\
	u|\pinf M = \varphi
	\end{cases}
		\end{equation}
for any continuous function $\varphi\colon \pinf M\to \R$.
\end{thm}

\begin{thm}
\label{main2}
Let $M$ be a Cartan-Hadamard manifold satisfying the curvature
assumptions \eqref{curv-bound-gen} and \eqref{A1}--\eqref{A7} in Section \ref{sec_bar_infty}. Furthermore, assume that the prescribed mean curvature function 
$H\colon M \to \R$ satisfies the assumptions \eqref{HVbound} and 
\eqref{asym_meancurv_assum} with a convex warping function $\varrho$ satisfying \eqref{rho+assumption2}, \eqref{rho_sign_assum}, and \eqref{rho_rad_assum}.
Then there exists a unique solution $u\colon M\to \R$
to the Dirichlet problem \eqref{ADP1}
for any continuous function $\varphi\colon \pinf M\to \R$.
\end{thm}

\begin{rem}
The following example illustrates the need of our assumptions about the warping function $\varrho$ in Theorems \ref{main1}  and \ref{main2}.  Let $N$ be the $(n+1)$-dimensional hyperbolic space $\mathbb{H}^{n+1}$ and consider the Killing vector field $X$ in $\mathbb{H}^{n+1}$ corresponding to a one-parameter family of parabolic isometries of $\mathbb{H}^{n+1}$ preserving a given  ideal point, say $p_0\in \partial_\infty\mathbb{H}^{n+1}$.  
This configuration cannot be directly compared with a rotationally invariant model (that is, invariant by a one-parameter family of \emph{elliptic} isometries) as we have assumed for instance in conditions \ref{rho_height_assum} and \ref{rho_height_assum1}.  This borderline case of a one-parameter family of parabolic isometries and the corresponding Killing field in $\mathbb{H}^{n+1}$ were studied by  Ripoll and Telichevsky in \cite{RT_BBMS} using different techniques relying on a variant of the Perron method. 
\end{rem}

\section{A priori height and gradient estimates}\label{sec-loc-estimates}
Throughout this section we denote by $\Omega_k = B(o,k)$ the geodesic ball centered at a given point $o\in M$ with radius $k\in\N$, and by $d(\cdot) = {\rm dist}(\cdot,\p\Omega_k)$ the distance function to the boundary of $\Omega_k$.


\subsection{Height estimate}

Fix $k\in\N$ and suppose that $u_k\in C^2(\Omega_k)$ is a solution of the Dirichlet problem \eqref{dir-aux}.
We aim to show that the function 
\begin{equation}
\label{h-bar}
v_k (x) = \sup_{\p\Omega_k} \varphi_k + h (d(x)),
\end{equation}
where $h$ will be determined later, is an upper barrier for the solution $u_k$. It suffices to show (see \cite[p. 795]{spruck} or \cite[pp. 239-240]{DHL}) that $v_k$ is a barrier in an open neighbourhood of $\p\Omega_k$ in which the points can be joined to $\p\Omega_k$ by unique geodesics. In this neighbourhood the distance function $d$ has the same regularity as $\p\Omega_k$ and therefore the derivatives of $d$ in the following computations are well-defined.

Since $X$ is Killing field, we have
\begin{equation}\label{kappa-def}
\langle \nabla \log \varrho, \nabla d\rangle = \frac{1}{\varrho} \ang{\nabla \varrho,\nabla d}= -\frac{1}{\varrho^2} \big\langle\bar\nabla_X X, \nabla d\rangle =: -\kappa(d),
\end{equation}
where $\kappa$ is the principal curvature of the Killing cylinder $\mathcal{C}_{k-d}$ over the geodesic sphere $\partial B(o, k-d)$. 
This implies that
\begin{align*}
\mathcal{Q}[v_k] &= {\rm div} \bigg(\frac{h'\nabla d}{\sqrt{\varrho^{-2}+ h'^2}}\bigg) -\kappa \frac{h'}{\sqrt{\varrho^{-2}+ h'^2}}\\
&= \frac{h'}{\sqrt{\varrho^{-2}+ h'^2}} \big(\Delta d - \kappa) + \p_d\bigg(\frac{h'}{\sqrt{\varrho^{-2}+ h'^2}}\bigg),
\end{align*}
where $\p_d$ denotes the derivative to the direction $\nabla d$.
However,
\[
\Delta d- \kappa = -nH_{k-d},
\]
where $H_{k-d}$ is the mean curvature of the cylinder $\mathcal{C}_{k-d}$, and we have
\begin{align*}
\p_d\bigg(\frac{h'}{\sqrt{\varrho^{-2}+ h'^2}}\bigg) &= \frac{h''}{\sqrt{\varrho^{-2}+ h'^2}} - \frac{h'}{(\varrho^{-2}+h'^2)^{3/2}} (\varrho^{-2}\kappa +  h'h'')  \\ 
&= \frac{\varrho^{-2}}{(\varrho^{-2}+h'^2)^{3/2}} (h''-\kappa h').
\end{align*}
Hence it follows that
\begin{align*}
\mathcal{Q}[v_k]  = -\frac{h'}{\sqrt{\varrho^{-2}+ h'^2}} nH_{k-d}  + \frac{\varrho^{-2}}{(\varrho^{-2}+h'^2)^{3/2}} (h''-\kappa h').
\end{align*}

Suppose that the principal curvature of the Killing cylinder $\mathcal{C}_{k-d}$ satisfies 
\[
\kappa(d) \ge -\frac{\varrho_0'(d)}{\varrho_0(d)},
\]
where $\varrho_0$ is a smooth positive increasing function on $[0,\infty)$. We note already at this point that, in the case of Cartan-Hadamard manifolds, $\nabla d = -\nabla r$ and this agrees with the assumption \eqref{rho_height_assum}. Then define the function $h$ as
\begin{equation}
h(d) = C\int_0^d\varrho_0^{-1}(t){{\rm d}t}
\end{equation}
for some constant $C>0$ to be fixed later. Now, since $h'>0$, we have
\[
h''-\kappa h' \le 0
\]
and
\begin{eqnarray*}
\mathcal{Q}[v_k]  \le -\frac{h'}{\sqrt{\varrho^{-2}+ h'^2}} nH_{k-d}  = 
-\frac{C\varrho}{\sqrt{\varrho^2_0 +C^2\varrho^2}} nH_{k-d}.
\end{eqnarray*}
Assuming that
\[
|H| < H_{k-d}
\]
in $\bar\Omega_k$ 
and choosing the constant $C$ as
\[
C^2 > \frac{H^2/H_{k-d}^2}{1-H^2/H_{k-d}^2} \frac{\sup \varrho_0^2}{\inf \varrho^2}
\]
we see that
\[
\mathcal{Q}[v_k] - nH \le 0
\]
and hence $v_k$ is an upper barrier for $u_k$.

Similarly we see that the function
\[
v_k^- = \inf_{\p\Omega_k} \varphi_k - h(d)
\]
is a lower barrier for $u_k$ and together these barriers give the following height estimate.

\begin{lem}\label{height-est-lem}
Assume that
\begin{equation}\label{|H|bound}
|H| < H_{k-d}
\end{equation}
in $\bar\Omega_k$ and that $u_k$ is a solution to the Dirichlet problem \eqref{dir-aux}. Then there exists a constant $C=C(\Omega_k)$ such that 
	\[
	\sup_{\Omega_k} |u_k| \le C + \sup_{\p\Omega_k} |\varphi|.
	\]
\end{lem}


\subsection{Boundary gradient estimate}

For given $\ve>0$ we define an annulus
	\[
	U_k(\ve) = \{x\in\Omega_k \colon d(x)<\ve\}.
	\]
In order to obtain a boundary gradient estimate, we aim to show that a function of the form
\[
w(x) = g(d(x)) + \psi(x)
\]
is an upper barrier in the set $U_k(\ve)$ for a fixed $\ve\in(0,1/2)$
 chosen so that $d$ is smooth in $U_k(\ve)$. Here we denote by $\psi$ the extension of the boundary data that is constant along geodesics issuing perpendicularly from $\partial\Omega_k$, i.e. $\psi(\exp_y t\nabla d(y)) = \varphi(y)$, where 
 $y\in\p\Omega_k$ and $\nabla d(y)$ is the unit inward normal to $\p\Omega_k$ at $y$. From \eqref{non-div-0} we have that
\begin{equation}
\label{non-div}
\mathcal{Q}[w] = \frac{1}{W}\Delta w-\frac{1}{W^3}\langle\nabla_{\nabla w}\nabla w, \nabla w\rangle -\frac{1}{\varrho^2} \bigg(1+\frac{1}{\varrho^2 W^2}\bigg)\Big\langle \bar\nabla_X X, \frac{\nabla w}{W}\Big\rangle,
\end{equation}
where 
\[
W = \sqrt{\varrho^{-2}+ |\nabla w|^2} = \sqrt{\varrho^{-2}+ g'^2 + |\nabla\psi|^2}.
\]
Since
\[
\nabla w = g'\nabla d +\nabla\psi, 
\]
with $\langle \nabla d, \nabla \psi\rangle=0$, it follows that
\[
\Delta w = g' \Delta d + g'' + \Delta\psi
\]
and 
\[
\langle\nabla_{\nabla w}\nabla w, \nabla w\rangle = 
g'^2 g''  - g'\langle \nabla_{\nabla \psi} \nabla d, \nabla\psi\rangle
+ \langle\nabla_{\nabla\psi}\nabla\psi, \nabla\psi\rangle.
\]
Moreover, by \eqref{kappa-def}
\[
\langle\nabla w, \bar\nabla_X X\rangle=g'\langle\nabla d, \bar\nabla_X X\rangle + \langle \nabla \psi, \bar\nabla_X X\rangle=g' \varrho^2\kappa + \langle \nabla \psi, \bar\nabla_X X\rangle.
\]
Using the expression \eqref{non-div} we obtain that 
\begin{align*}
\mathcal{Q}[w]   &=  \frac{1}{W}(g''+ g'\Delta d +\Delta\psi)-\frac{1}{W^3} (g'^2 g'' -g'\langle \nabla_{\nabla \psi} \nabla d, \nabla\psi\rangle+\langle \nabla_{\nabla \psi}\nabla \psi, \nabla\psi\rangle)\\
&\quad - \frac{1}{W}\bigg(1+\frac{1}{\varrho^2 W^2}\bigg)\bigg( g'\kappa + \Big\langle \bar\nabla_{\frac{X}{|X|}} \frac{X}{|X|}, \nabla\psi\Big\rangle\bigg)
\end{align*}
and combining with the previous reasoning, this results to
\begin{align*}
\mathcal{Q}[w]   &=  \frac{g''}{W^3}(\varrho^{-2}+|\nabla \psi|^2) + \frac{g'}{W}\bigg(\Delta d- \bigg(1+\frac{1}{\varrho^2 W^2}\bigg)\kappa \bigg) + \frac{1}{W}\Delta\psi \\ &\quad-\frac{1}{W^3} \langle \nabla_{\nabla \psi}\nabla \psi, \nabla\psi\rangle + \frac{g'}{W^3}\langle \nabla_{\nabla \psi} \nabla d, \nabla\psi\rangle \\
&\quad - \frac{1}{W}\bigg(1+\frac{1}{\varrho^2 W^2}\bigg) \Big\langle \bar\nabla_{\frac{X}{|X|}} \frac{X}{|X|}, \nabla\psi\Big\rangle.
\end{align*}
We note that
\begin{align*}
\frac{1}{W}\Delta\psi -\frac{1}{W^3} \langle \nabla_{\nabla \psi}\nabla \psi, \nabla\psi\rangle =
\frac{1}{W} \left(g^{ij} - \frac{\psi^i\psi^j}{W^2}\right)\psi_{i;j}
\end{align*}
and, on the other hand,
\begin{align*}
\frac{1}{W} \left(g^{ij} - \frac{(g'd^i + \psi^i)(g'd^j + \psi^j)}{W^2}\right) &\psi_{i;j} =
\frac{1}{W} \left(g^{ij} - \frac{\psi^i\psi^j}{W^2}\right)\psi_{i;j} -\frac{2g'd^i\psi^j}{W^3}\psi_{i;j} \\ 
&= \frac{1}{W} \left(g^{ij} - \frac{\psi^i\psi^j}{W^2}\right)\psi_{i;j} - \frac{2g'}{W^3}\ang{\nabla_{\nabla d}\nabla\psi, \nabla\psi}.
\end{align*}
Moreover, the matrix $(\sigma^{ij})$ has eigenvalues $1/(\varrho^2W^3)$ and $1/W$ which can be estimated as
\begin{equation}\label{eigenv-est}
\max\left(\frac{1}{\varrho^2W^3}, \frac{1}{W} \right) \le \frac{1}{\varrho^2W^3} + \frac{1}{W} \le \frac{1}{W}(1+\varrho^2).
\end{equation}
When $\varrho\ge1$, this is trivial, and when $\varrho<1$ we can choose the constant $K$ in the definition \eqref{defig} of $g$ such that this holds. Therefore we are able to estimate
\begin{align*}
\mathcal{Q}[w]&\le   \frac{g''}{W^3}(\varrho^{-2}+|\nabla \psi|^2) + \frac{g'}{W}\bigg(\Delta d- \bigg(1+\frac{1}{\varrho^2 W^2}\bigg)\kappa \bigg) + \frac{1}{W}(1+\varrho^2) ||\nabla^2 \psi||\\
&\quad + \frac{3g'}{W^3}\langle \nabla_{\nabla \psi} \nabla d, \nabla\psi\rangle+ \frac{1}{W}\bigg(1+\frac{1}{\varrho^2 W^2}\bigg) \langle \nabla\log\varrho, \nabla\psi\rangle \\
&\le   \frac{g''}{W^3}(\varrho^{-2}+|\nabla \psi|^2) - \frac{g'}{W}\bigg(nH_{k-d}+\frac{1}{\varrho^2 W^2}\kappa \bigg) 
+ \frac{1}{W}(1+\varrho^2) ||\nabla^2 \psi||\\
&\quad + \frac{3g'}{W^3} | II_{k-d}(\nabla \psi, \nabla\psi)|+ \frac{1}{W}\bigg(1+\frac{1}{\varrho^2 W^2}\bigg) \langle \nabla\log\varrho, \nabla\psi\rangle.
\end{align*}

Now we choose
\begin{equation}\label{defig}
g(d) = \frac{C}{\log(1+K)} \log (1+Kd),
\end{equation}
where
\[
C = K(1+K\ve)\log(1+K)
\]
and $K\ge (1-2\ve)\ve^{-2}$ so large that 
\begin{equation}\label{K-u-dep}
C\ge 2\big(\max_{\bar\Omega_k} |u_k| + \max_{\bar\Omega_k} |\psi|\big).
\end{equation}
Note that this choice yields $g \ge u_k$ on the ``inner'' boundary $\{x\in\Omega_k\colon d(x)=\ve\}$ of $U_k(\ve)$.
Then, for $K$ large, we have
\[
1\ge \frac{g'^2}{W^2} \ge \frac{K^4(1+K\ve)^2}{(1+Kd)^2 L + K^4 (1+K\ve)^2} = 
\frac{(1+K\ve)^2}{\frac{(1+Kd)^2}{K^4} L + (1+K\ve)^2}\ge c_1^2>0,
\]
where 
\[
c_1^2 \le \frac{(1+K\ve)^2}{(1+K\ve)^2+L}
\]
with 
\[
L = \sup_{\Omega} (\varrho^{-2} + |\nabla\psi|^2).
\]
We also have
\[
g''(d) = -\frac{g'^2}{K(1+K\ve)},
\]
which implies that
\[
\frac{g''}{W^2} = -\frac{1}{K(1+K\ve)}\frac{g'^2}{W^2} \le - \frac{1}{K(1+K\ve)}c_1^2.
\]
Hence we obtain
\begin{align*}
&\mathcal{Q}[w]   \le  - \frac{1}{K(1+K\ve)}c_1^2 \frac{1}{W}(\varrho^{-2}+|\nabla \psi|^2) - c_1\bigg(nH_{k-d}+\frac{1}{\varrho^2 W^2}\kappa \bigg)\\
&\quad + \frac{1}{W}(1+\varrho^2) ||\nabla^2 \psi|| + \frac{g'}{W} \frac{|\nabla\psi|^2}{W^2} || II_{k-d} ||+ \frac{|\nabla\psi|}{W} |\nabla\log\varrho|\bigg(1+\frac{1}{\varrho^2 W^2}\bigg).
\end{align*}
However,
\[
\frac{1}{\varrho^2 W^2} \ge \frac{c_1^2}{\varrho^2 g'^2}
\ge \bigg(\frac{1+Kd}{K^2(1+K\ve)}\bigg)^2 \frac{c_1^2}{\varrho^2}\ge \frac{1}{K^4(1+K\ve)^2} \frac{c_1^2}{\varrho^2}
\]
and
\[
\frac{1}{W} \ge \frac{1+Kd}{K^2(1+K\ve)}c_1 \ge \frac{1}{K^2(1+K\ve)} c_1.
\]
Combining these with the fact that $W\ge K^2$, we obtain the estimate
\begin{align*}
\mathcal{Q}[w]   &\le  - c_1\frac{K+\kappa}{K^4(1+K\ve)^2}\frac{c_1^2}{\varrho^2} - c_1 nH_{k-d}+ \frac{1}{W}(1+\varrho^2) ||\nabla^2 \psi|| \\
&\quad + \frac{g'}{W} \frac{|\nabla\psi|^2}{W^2} || II_{k-d} ||+\bigg(1+\frac{1}{\varrho^2} \frac{1}{K^2}\bigg)\frac{|\nabla\psi|}{W} |\nabla\log\varrho| \\
&\le  - c_1\frac{K+\kappa}{K^4(1+K\ve)^2}\frac{c_1^2}{\varrho^2} - c_1 nH_{k-d}+ \frac{1}{K^2}(1+\varrho^2) ||\nabla^2 \psi|| \\
&\quad +  |\nabla\psi|^2 || II_{k-d} ||\frac{1}{K^4}+ \bigg(1+\frac{1}{\varrho^2K^4} \bigg) |\nabla\psi||\nabla\log\varrho|\frac{1}{K^2}.
\end{align*}
Therefore
\begin{align}\label{bdgrad-last-est}
\mathcal{Q}[w]   &-nH \le  -n(c_1H_{k-d}+H)  - c_1\frac{K+\kappa}{K^4(1+K\ve)^2}\frac{c_1^2}{\varrho^2} + \frac{1}{K^2}(1+\varrho^2) ||\nabla^2 \psi|| \nonumber \\
&+  |\nabla\psi|^2 || II_{k-d} ||\frac{1}{K^4}+ \bigg(1+\frac{1}{\varrho^2K^4} \bigg) |\nabla\psi||\nabla\log\varrho|\frac{1}{K^2}.
\end{align}
Finally observe that
\[
c_1 H_{k-d} \ge |H|
\]
if we choose $K$ such that
\[
\frac{H^2}{H_{k-d}^2} \le c_1^2 \le \frac{(1+K\ve)^2}{(1+K\ve)^2 +L},
\]
that is
\begin{equation}\label{K-curv-cond}
(1+K\ve)^2 \ge L\frac{H^2/H_{k-d}^2}{1-H^2/H_{k-d}^2}.
\end{equation}
Taking \eqref{eigenv-est}, \eqref{K-u-dep}, \eqref{bdgrad-last-est} and \eqref{K-curv-cond} into account, we can choose 
\[
K=K(\Omega_k,H,||\psi||_{C^2},\ve,\sup_{\Omega_k}|u|)
\]
so large that
\[
\mathcal{Q}[w]-nH \le 0
\]
holds in $U_k(\ve)$. This suffices for the following boundary gradient estimate.
\begin{lem}\label{bdgrad-lem}
Assume that
\[
|H| < H_{k-d}
\]
in $\bar\Omega_k$ and that $u_k$ is a solution to the Dirichlet problem \eqref{dir-aux}. Then there exists a constant $C=C(\Omega_k,H,||\psi||_{C^2},\ve,\sup_{\Omega_k}|u|)$ such that 
\[
\max_{\p\Omega_k} |\nabla u_k| \le C.
\]
\end{lem}


\subsection{Interior gradient estimate}
In this subsection we prove a quantitative interior gradient estimate
that is interesting on its own.
The proof is based on the technique due to Korevaar and Simon \cite{korevaar},
and further developed by Wang \cite{WangX}. We will perform the computations in a coordinate free way.

Let $u$ be a ($C^3$-smooth) positive solution of the equation \eqref{PDE} in a ball $B(p,R) \subset M.$ Suppose that sectional curvatures in $B(p,R)$ are bounded from below by $-K_0^2$ for some constant $K_0=K_0(p,R)\ge 0$.
We consider a nonnegative and smooth function $\eta$ with $\eta=0$ in $M\setminus B(p,R)$ and define a function $\chi$ in $B(p,R)$ of the form 
\begin{equation}\label{chi-def}
\chi = \eta \gamma(u) \psi (|\nabla u|^2),
\end{equation}
where the functions $\eta$, $\gamma$ and $\psi$ will be specified later.

Suppose that $\chi$ attains its maximum at $x_0\in B(p,R)$, and without loss of generality, that $\eta(x_0)\neq 0$. 
Then at $x_0$
	\begin{equation}\label{Foc-1}
		(\log \chi)_j = \frac{\eta_j}{\eta} + \frac{\gamma'}{\gamma} u_j + 2\frac{\psi'}{\psi} u^k u_{k;j} =0
	\end{equation}
and therefore
	\begin{equation}\label{reduction}
		2\frac{\psi'}{\psi}u^k u_{k;j} = -\bigg(\frac{\eta_j}{\eta} + \frac{\gamma'}{\gamma} u_j\bigg).
	\end{equation}
Moreover, the matrix
	\begin{align*}
	(\log\chi)_{i;j} = (\log \eta)_{i;j} &+\bigg(\frac{\gamma'}{\gamma} \bigg)' u_iu_j +\frac{\gamma'}{\gamma} u_{i;j}
	  + 2\frac{\psi'}{\psi} (u^k u_{k;ij}+ u^k_{;i} u_{k;j}) \\
	&+ 4\bigg(\frac{\psi'}{\psi}\bigg)' u^k u_{k;i} u^\ell u_{\ell;j}
	\end{align*}
is non-positive at $x_0$. Applying the Ricci identities for the Hessian of $u$ we have
	\[
		u_{k;ij} = u_{i;kj} = u_{i;jk} + R^\ell_{kji}u_\ell,
	\]
and this yields
	\begin{align*}\label{log-matr}
	(\log\chi)_{i;j} 
	&\,\,= \frac{\eta_{i;j}}{\eta} +\frac{\gamma''}{\gamma}  u_iu_j +\frac{\gamma'}{\gamma} u_{i;j} + \frac{\gamma'}{\gamma}\bigg(\frac{\eta_i}{\eta} u_j + \frac{\eta_j}{\eta}u_i\bigg) \\
	&\,\,\,\,+ 2\frac{\psi'}{\psi} (u^k u_{i;jk}+ u^k_{;i} u_{k;j}) -2\frac{\psi'}{\psi}R_{jki}^\ell u^k u_\ell + 
	4\bigg(\bigg(\frac{\psi'}{\psi}\bigg)' - \frac{\psi'^2}{\psi^2}\bigg) u^k  u_{k;i}  u^\ell u_{\ell;j}.   \nonumber
	\end{align*}
On the other hand, denoting
	\begin{equation*}
		f(x) = nHW - \langle \nabla\log\varrho, \nabla u\rangle \bigg(1+\frac{1}{\varrho^2 W^2}\bigg),
	\end{equation*}
and differentiating both sides in \eqref{non-div-0} we have
	\begin{align}\label{sigma-u-ijk}
		\sigma^{ij} u_{i;jk} = f_k - \sigma^{ij}_{;k} u_{i;j}.
	\end{align}
Contracting \eqref{sigma-u-ijk} with $u^k$, we get
	\begin{align*}
	\sigma^{ij} u^k u_{i;jk}  &= f_k u^k +\frac{1}{W^2} u^k(u^i_{;k} u^j + u^i u^j_{;k})  u_{i;j} \\
		\quad& - \frac{2}{W^4}  u^i u^j u_{i;j} (-\varrho^{-2}(\log\varrho)_k u^k+ u^k u^\ell u_{\ell;k}).
	\end{align*}
Using the previous identity, \eqref{reduction} 
and noticing that 
	\begin{align*}
	\sigma^{ij} R_{jki}^\ell u^k u_\ell 
	&= -{\rm Ric}_g (\nabla u, \nabla u),
	\end{align*}
lengthy computations give
	\begin{align*}
	0 &\ge \sigma^{ij}(\log \chi)_{i;j}=  2n\frac{\psi'}{\psi}\langle \nabla H, \nabla u\rangle W + nH\frac{\gamma'}{\gamma}
	\frac{1}{\varrho^2 W} - nH \frac{1}{W}\bigg\langle \frac{\nabla \eta}{\eta}, \nabla u\bigg\rangle \\
&\quad	- 2nH \frac{1}{\varrho^2 W}\frac{\psi'}{\psi}\langle\nabla\log\varrho, \nabla u\rangle 
 + 4\bigg(\bigg(\frac{\psi'}{\psi}\bigg)' - \frac{\psi'^2}{\psi^2}+\frac{3}{2}\frac{\psi'}{\psi}\frac{1}{W^2}\bigg)
	 \sigma^{i\ell}  u^j u^k  u_{k;i}   u_{j;\ell}\\
&\quad	 + \sigma^{ij} \frac{\eta_{i;j}}{\eta} 
	 + 2\frac{\gamma'}{\gamma}\frac{1}{\varrho^2 W^2}\bigg\langle\frac{\nabla\eta}{\eta}, \nabla u\bigg\rangle 
 +2\frac{\psi'}{\psi} \big({\rm Ric}_g (\nabla u, \nabla u)-\nabla^2\log\varrho(\nabla u, \nabla u)\big)  \\
&\quad +\frac{4|\nabla u|^2}{\varrho^2 W^4}\frac{\psi'}{\psi}\langle\nabla\log\varrho, \nabla u\rangle^2 
	- \frac{2}{\varrho^2 W^4}\langle\nabla\log\varrho, \nabla u\rangle \bigg\langle \frac{\nabla\eta}{\eta}
	+ \frac{\gamma'}{\gamma}\nabla u, \nabla u\bigg\rangle\\
&\quad - 2\frac{\psi'}{\psi}\frac{1}{\varrho^2 W^2} \nabla^2 \log \varrho(\nabla u, \nabla u) 
	+\bigg\langle \nabla \log\varrho, \frac{\nabla \eta}{\eta}\bigg\rangle\bigg(1+\frac{1}{\varrho^2 W^2}\bigg) \\
&\quad - \frac{2}{\varrho^2W^4}  \bigg\langle \frac{\nabla\eta}{\eta}
	+\frac{\gamma'}{\gamma}\nabla u, \nabla u\bigg\rangle \langle \nabla \log\varrho, \nabla u\rangle 
	+\frac{\gamma''}{\gamma} \sigma^{ij}u_i u_j+2\frac{\psi'}{\psi}\sigma^{i\ell} \sigma^{jk} u_{k;i} u_{j;\ell}.
	\end{align*}
Notice that \eqref{reduction} yields to
\begin{align*}
	4\frac{\psi'^2}{\psi^2}\sigma^{i\ell}  u^j u^k  u_{k;i}   u_{j;\ell} &= \bigg|\frac{\nabla\eta}{\eta} 
	+\frac{\gamma'}{\gamma}\nabla u\bigg|^2_\sigma \ge \frac{1}{\varrho^2 W^2} \bigg|\frac{\nabla\eta}{\eta}
	+\frac{\gamma'}{\gamma}\nabla u\bigg|^2_g \\
&=\frac{|\nabla u|^2}{\varrho^2 W^2} \bigg|\frac{\nabla\eta}{|\nabla u|\eta}
	+ \frac{\gamma'}{\gamma}\frac{\nabla u}{|\nabla u|}\bigg|^2_g.
\end{align*}
Plugging this into the previous estimate, we get
\begin{align*}
\frac{\psi^2}{\psi'^2} &\bigg(\bigg(\frac{\psi'}{\psi}\bigg)' - \frac{\psi'^2}{\psi^2}
	+\frac{3}{2}\frac{\psi'}{\psi}\frac{1}{W^2}\bigg) \frac{|\nabla u|^2}{\varrho^2 W^2} 
	\bigg|\frac{\nabla\eta}{|\nabla u|\eta} +\frac{\gamma'}{\gamma}\frac{\nabla u}{|\nabla u|}\bigg|^2_g \\ 
&\quad \quad \quad +\frac{\gamma''}{\gamma} \sigma^{ij}u_i u_j
	+2\frac{\psi'}{\psi}\sigma^{i\ell} \sigma^{jk} u_{k;i} u_{j;\ell}  \\
&\le  2n\frac{\psi'}{\psi} | \nabla H| | \nabla u| W+2n\frac{\psi'}{\psi}|H| \frac{|\nabla u |}{W}\frac{|\nabla\log\varrho|}{\varrho^2}\\
&	-2\frac{\psi'}{\psi} \big({\rm Ric}_g (\nabla u, \nabla u)-\nabla^2\log\varrho(\nabla u, \nabla u)\big)  
+ 4\frac{\psi'}{\psi}\frac{|\nabla\log\varrho|^2}{\varrho^2} \frac{|\nabla u|^4}{W^4}\\
&	+ 2\frac{\psi'}{\psi}\frac{|\nabla^2 \log \varrho|}{\varrho^2}\frac{|\nabla u|^2}{W^2}+n|H|\frac{\gamma'}{\gamma}\frac{1}{\varrho^2 W}
	+ n|H| \bigg|\frac{\nabla \eta}{\eta}\bigg|\frac{|\nabla u|}{W} - \sigma^{ij} \frac{\eta_{i;j}}{\eta} 
	 \\
&+ 2\frac{\gamma'}{\gamma}\frac{1}{\varrho^2 W}\frac{|\nabla u|}{W}\bigg|\frac{\nabla\eta}{\eta}\bigg|+ 4\frac{|\nabla\log\varrho|}{\varrho^2} \bigg(\bigg|\frac{\nabla\eta}{\eta}\bigg| \frac{|\nabla u|^2}{W^4}
	+ \frac{\gamma'}{\gamma}\frac{|\nabla u|^3}{W^4}\bigg)\\
&	+|\nabla \log\varrho|\bigg|\frac{\nabla \eta}{\eta}\bigg|  \bigg(1+\frac{1}{\varrho^2 W^2}\bigg).
\end{align*}

Suppose that $|\nabla u|(x_0)>1$. Otherwise we are done.
Hence, following \cite{WangX}, we set
\begin{equation}
\psi(t) = \log t,
\end{equation}
where $t=|\nabla u|^2$. Then we have
	\begin{align*}
	&\frac{|\nabla u|^2}{W^2}\frac{\psi^2}{\psi'^2} \bigg(\bigg(\frac{\psi'}{\psi}\bigg)' - \frac{\psi'^2}{\psi^2}
	+\frac{3}{2}\frac{\psi'}{\psi}\frac{1}{W^2}\bigg)
 = \frac{t}{W^2}\bigg( \log t \frac{\frac{1}{2}t-\varrho^{-2}}{t+\varrho^{-2}} -2 \bigg).
	\end{align*}
Now we fix a constant 
\[
\max\left\{\frac{2}{3}, \frac{\varrho^2}{1+\varrho^2}\right\}< \beta <1
\]
 and suppose that
	\begin{equation}\label{estnabla1}
		\frac{t}{W^2} =\frac{|\nabla u|^2}{W^2} \ge \beta.
	\end{equation}
Setting $\frac{1}{\varrho^2}\frac{\beta}{1-\beta} =: e^{\delta'}$, 
$
\delta = \frac{3}{2}\beta-1,$ 
and
	$
		\mu := 2\beta\frac{\delta\delta'-2}{\delta'},
	$
we get
\begin{align*}
	&\mu \log |\nabla u|\frac{1}{\varrho^2} \bigg|\frac{\nabla\eta}{|\nabla u|\eta}
	+ \frac{\gamma'}{\gamma}\frac{\nabla u}{|\nabla u|}\bigg|^2_g  
	+\frac{\gamma''}{\gamma} \frac{|\nabla u|^2}{\varrho^2 W^2} \\
&\quad\quad\quad + 2\frac{\psi'}{\psi} |\nabla u|^2\bigg(\Ric_g \bigg(\frac{\nabla u}{|\nabla u|}, \frac{\nabla u}{|\nabla u|}\bigg)-\nabla^2\log\varrho\bigg(\frac{\nabla u}{|\nabla u|}, \frac{\nabla u}{|\nabla u|}\bigg)\bigg)  \\
&\,\, 
\le  \frac{2}{\sqrt\beta \delta'} \big(n|\nabla H|+ (1-\beta)n|H| |\nabla \log\varrho|+2(1-\beta)|\nabla\log\varrho|^2+(1-\beta)|\nabla^2\log\varrho|\big)\\
& \,\,\,\, +\sqrt{1-\beta} \frac{1}{\varrho}\frac{\gamma'}{\gamma} \big(n|H|+4|\nabla\log\varrho|\big)+2\sqrt{1-\beta}\frac{1}{\varrho }\frac{\gamma'}{\gamma}\bigg|\frac{\nabla\eta}{\eta}\bigg|\\ 
&\,\,\,\,
	+  \bigg|\frac{\nabla\eta}{\eta}\bigg|\big(n|H|+(6-5\beta)|\nabla\log\varrho|\big) -\sigma^{ij}\frac{\eta_{i;j}}{\eta}.
\end{align*}
 By modifying the argument in \cite[Proof of Theorem 4.1, Case 2]{RSS} we may assume 
that the maximum point $x_0$ is not in the cut-locus $C(p)$ of $p$. 
Then we choose $\eta$ as
\begin{equation}
\eta = \hat\eta^2
\end{equation}
where
\begin{equation}
\hat\eta = 1- \frac{1}{C_R} \int_0^r \xi(\tau)\, {\rm d}\tau,
\quad r=d(\cdot,p),
\end{equation}
with
\[
C_R = \int_0^R \xi(\tau)\, {\rm d}\tau
\]
and $\xi(\tau)=K_0^{-1}\sinh(K_0\tau)$ if $K_0>0$ and $\xi(\tau)=\tau$ if $K_0=0$.
Denoting
\begin{equation*}
\kappa = \varrho^{-2}\langle \bar\nabla_X\bar\nabla r, X\rangle = \langle \nabla r, \nabla \log\varrho\rangle , 
\end{equation*}
one can show that $|\nabla \eta| = 2\hat \eta \frac{\xi(r)}{C_R}$ and
\begin{eqnarray*}
& & \Delta_\Sigma \eta = 2 \hat\eta \Delta_\Sigma \hat\eta + 2 |\nabla^\Sigma\hat\eta|^2 \\
& & \,\, \le  2\hat\eta(r)\frac{\xi(r)}{C_R}\bigg| ( n-1)\frac{\xi'(r)}{\xi(r)}+\kappa + n|H| +(1-\beta)\bigg|\frac{\xi'(r)}{\xi(r)}-\kappa\bigg|\bigg|+ 2 \frac{\xi^2(r)}{C_R^2}.
\end{eqnarray*}
As in \cite{WangX},  we set
\begin{equation*}
\gamma (u) = 1 +\frac{1}{M} (\min_{\bar{B}(p,R)}\varrho)u 
\end{equation*}
where $M>0$ is a constant to be fixed later. Then $\gamma''=0$ and
hence
\begin{align}
\label{ineq-almost}
	&\mu \log |\nabla u|\frac{1}{\varrho^2} \bigg|\frac{\nabla\eta}{|\nabla u|\eta}
	+ \frac{\gamma'}{\gamma}\frac{\nabla u}{|\nabla u|}\bigg|^2_g  
\nonumber\\
&\quad\quad\quad + 2\frac{\psi'}{\psi} |\nabla u|^2\bigg(\Ric_g \bigg(\frac{\nabla u}{|\nabla u|}, \frac{\nabla u}{|\nabla u|}\bigg)-\nabla^2\log\varrho\bigg(\frac{\nabla u}{|\nabla u|}, \frac{\nabla u}{|\nabla u|}\bigg)\bigg)  \le  \widetilde M\frac{1}{M\eta},
\end{align}
where
\begin{align}\label{mtilde}
& \widetilde M=\frac{2}{\sqrt\beta \delta'} \big(n|\nabla H|+ (1-\beta)n|H| |\nabla \log\varrho|+2(1-\beta)|\nabla\log\varrho|^2
\nonumber\\
& \,\,\,\, +(1-\beta)|\nabla^2\log\varrho|\big)M\eta
+\sqrt{1-\beta}  \big(n|H|+4|\nabla\log\varrho|\big)\eta +4\sqrt{1-\beta} \frac{\xi(r)}{C_R}\hat{\eta}\nonumber\\
&\,\,\,\,
	+ 2\frac{\xi(r)}{C_R} \big(n|H|+(6-5\beta)|\nabla\log\varrho|\big)M \hat{\eta}\\
	& \,\,\,\, + M\bigg(2\hat\eta(r)\frac{\xi(r)}{C_R}\bigg| ( n-1)\frac{\xi'(r)}{\xi(r)}+\kappa + n|H| +(1-\beta)\bigg|\frac{\xi'(r)}{\xi(r)}-\kappa\bigg|\bigg|+ 2 \frac{\xi^2(r)}{C_R^2}\bigg).\nonumber
\end{align}
Let $L = L(p,R)\ge 0$ be chosen in such a way that
\begin{equation}\label{Ldef}
{\rm Ric}_g + \nabla^2\log\varrho \ge -L g
\end{equation}
in $B(p,R)$. Then 
we obtain 
\begin{align*}
	&\mu \log |\nabla u|\frac{1}{\varrho^2} \bigg|\frac{\nabla\eta}{|\nabla u|\eta}
	+ \frac{\gamma'}{\gamma}\frac{\nabla u}{|\nabla u|}\bigg|^2_g   - 2 L \frac{1}{\delta'} \le  \widetilde M\frac{1}{M\eta}.
\end{align*}
Set $
M = \max_{\bar{B}(p,R)} u$. 
We consider first the case
	\[
	\left| \frac{\nabla \eta}{|\nabla u| \eta} \right| \le \frac{\gamma'}{2\gamma}.
	\]
Then we have
\begin{align*}
	\eta \log |\nabla u| \le  \frac{4\gamma^2 \varrho^2}{\mu \min_{\bar{B}(p,R)}\varrho^2} \bigg(\widetilde M M + 2LM^2\frac{\eta}{\delta'}\bigg).
\end{align*}
On the other hand, when 
	\[
	\frac{\gamma'}{2\gamma} \le \left| \frac{\nabla \eta}{|\nabla u| \eta} \right|
	\]
we have
\begin{equation*}
\eta |\nabla u| \le \frac{4\gamma }{\gamma^\prime}\frac{\xi(r)}{C_R}.
\end{equation*}
which implies that
\begin{equation*}
\eta \log|\nabla u| \le \frac{4\gamma }{\gamma^\prime }\frac{\xi(r)}{C_R}.
\end{equation*}
Hence at $x_0$
\begin{equation}
\eta \log|\nabla u|  \le 
\max\left\{\frac{4\gamma(u(x_0))\xi(r(x_0))}{\gamma'(u(x_0)) C_R}, \frac{4\gamma^2(u(x_0))\varrho^2(x_0)}{\mu\min_{\bar{B}(p,R)}\varrho^2}\bigg(\widetilde M M + 2LM^2\frac{1}{\delta'}\bigg)\right\}.
\end{equation}
Since $\eta(p)=1$ and $\gamma(p)\ge 1$ we conclude that 
\begin{align}\label{intgradest}
&\log |\nabla u(p)| \le \eta(p)\gamma(p)\log |\nabla u(p)| 
\le \eta(x_0)\gamma(x_0)\log|\nabla u(x_0)|
\\
&\le 
\frac{4M(1+\min_{\bar{B}(p,R)}\varrho)^2}{\min_{\bar{B}(p,R)}\varrho}
\max\left\{\frac{\xi(r(x_0))}{C_R},\frac{(1+\min_{\bar{B}(p,R)}\varrho)\varrho^2(x_0)}{\mu\min_{\bar{B}(p,R)}\varrho}\bigg(\widetilde M  + 2LM\frac{1}{\delta'}\bigg)\right\}\nonumber
\end{align}
unless $|\nabla u(x_0)|\le 1$. 

We have proven the following quantitative gradient estimate. Here we denote by ${\mathcal{R}_B}$ the Riemannian curvature tensor in a set $B$.

\begin{lem}\label{int-grad-est}
Let $u$ be a positive solution of \eqref{PDE} in an open set $\Omega$ and let $B=B(p,R)\subset\Omega$. Then there exists a constant
$
C=C({\mathcal{R}}_{B},\varrho|B,H|B,u(p),\max_{\bar{B}}u,R)
$
such that
	\begin{equation*}
	|\nabla u(p)| \le C.
	\end{equation*}
\end{lem}
If the gradient of $u$ is continuous up to the boundary of $\Omega$ and $\Omega$ is bounded, we obtain the following quantitative global estimate.
\begin{lem}\label{glob-int-grad-est}
Let $u$ be a positive solution of \eqref{PDE} in a bounded open set $\Omega$
and suppose, moreover, that $u\in C^1(\bar{\Omega})$. Then there exists a constant
\[
C=C({\mathcal{R}}_{\Omega},\varrho|\Omega,H|\Omega,u(p),\max_{\bar{\Omega}}u,\diam(\Omega),\max_{\partial\Omega}|\nabla u|)
\]
such that
	\begin{equation*}
	|\nabla u(p)| \le C
	\end{equation*}
	for every $p\in\bar{\Omega}$.
\end{lem}

\begin{proof}
Let $p\in\Omega$ and $R=\diam(\Omega)$.
Define in $\bar{\Omega}\cap B(p,R)$ a function 
\[
\chi = \eta \gamma(u) \psi (|\nabla u|^2),
\]
where $\eta,\ \gamma$, and $\psi$ are as in the previous proof. If 
$\chi$ attains its maximum in an interior point 
$x_0\in B(p,R)\cap\Omega$, the proof of Lemma~\ref{int-grad-est} applies and we have a desired upper bound. Otherwise, $\chi$ attains its maximum at $x_0\in\partial\Omega$, but then 
$|\nabla u(x_0)|\le\max_{\partial\Omega}|\nabla u|$
and again we are done.
\end{proof}

We remark that a global gradient estimate for bounded Killing graphs  follows immediately from \eqref{intgradest}, \eqref{mtilde}, and \eqref{Ldef} in the case of bounded warping functions under some assumptions on the curvature.

\begin{cor}
\label{global-grad}
 Suppose that the sectional curvatures in $M$ satisfy $K_M \ge -K_0$ for some positive constant $K_0$. Suppose also that $\inf_M\varrho >0$ and that $|| \varrho||_{C^2(M)} <+\infty$.  If a function $u:M\to \mathbb{R}$ is uniformly bounded and the mean curvature of its graph satisfies $||H||_{C^1(M)}<+\infty$ then the gradient of $u$ is uniformly bounded.
\end{cor}

\section{Global barriers}
In this section we present two methods to obtain global (upper and lower) barriers for solutions to \eqref{dir-aux}.

In the case when $\widetilde H$ is constant along flow lines of $X$, that is, when $\widetilde H$ is a function in $M$, there is a conservation law (a flux formula) corresponding to the invariance of $\mathcal{A}_{\widetilde H}$ with respect to the flow generated by $X$.  This flux formula for graphs is stated as
\begin{equation}
\label{flux}
\int_{\Gamma} \Big\langle \frac{\nabla u}{W}, \nu\Big\rangle \varrho \,{\rm d}\Gamma =  \int_\Omega n\widetilde H\varrho \, {\rm d}M,
\end{equation}
where $\Gamma = \partial\Omega$ and $\nu$ is the outward unit normal vector field along $\Gamma \subset M$. 

Suppose for a while that $M$ is a model manifold with respect to a fixed pole $o\in M$ and that $\varrho=|X|$ is a radial function. In terms of polar coordinates $(r, \vartheta)\in \mathbb{R}^+\times\mathbb{S}^{n-1}$ centered at $o$ the metric in $M$ is of the form
\[
g = {\rm d}r^2 + \xi^2(r)\, {\rm d}\vartheta^2,
\]
where ${\rm d}\vartheta^2$ stands for the canonical metric in $\mathbb{S}^{n-1}$. Suppose that $\widetilde H$ and $u$ are also  radial functions. Applying (\ref{flux}) to $\Omega=B(o,r)$, the geodesic ball centered at $o$ with radius $r$, we obtain
\begin{equation}
\label{flux-2}
\frac{u'(r)}{\sqrt{\varrho^{-2}(r)+u'^2(r)}} \varrho(r) \xi^{n-1} (r) = \int_0^r n\widetilde H(\tau) \varrho(\tau) \xi^{n-1}(\tau) \, {\rm d}\tau
\end{equation}
This is a first integral of (\ref{PDE}) in this rotationally invariant setting. Indeed, taking derivatives on both sides of (\ref{flux-2}) with respect to $r$ we get
\begin{eqnarray*}
& & n\widetilde H(r)=\bigg(\frac{u'(r)}{\sqrt{\varrho^{-2}(r)+u'^2(r)}} \bigg)' + \frac{u'(r)}{\sqrt{\varrho^{-2}(r)+u'^2(r)}}  \bigg(\frac{\varrho'(r)}{\varrho(r)} + (n-1)\frac{\xi'(r)}{\xi(r)}\bigg).
\end{eqnarray*}
On the other hand in this particular setting (\ref{PDE})  becomes
\begin{eqnarray*}
& & n\widetilde H(r) = {\rm div}\bigg(\frac{u'(r)}{\sqrt{\varrho^{-2}(r)+u'^2(r)}} \partial_r\bigg) +\bigg(\frac{u'(r)}{\sqrt{\varrho^{-2}(r)+u'^2(r)}} \bigg) \frac{\varrho'(r)}{\varrho(r)}\\
& &\,\, =  \bigg(\frac{u'(r)}{\sqrt{\varrho^{-2}(r)+u'^2(r)}} \bigg)' + \frac{u'(r)}{\sqrt{\varrho^{-2}(r)+u'^2(r)}} \,{\rm div}\,\partial_r\\
& &\quad  +\bigg(\frac{u'(r)}{\sqrt{\varrho^{-2}(r)+u'^2(r)}} \bigg) \frac{\varrho'(r)}{\varrho(r)}\\
& & \,\, = \bigg(\frac{u'(r)}{\sqrt{\varrho^{-2}(r)+u'^2(r)}}\bigg)' + \frac{u'(r)}{\sqrt{\varrho^{-2}(r)+u'^2(r)}} \bigg( (n-1)\frac{\xi'(r)}{\xi(r)} +\frac{\varrho'(r)}{\varrho(r)}\bigg).
\end{eqnarray*}
It is convenient to write (\ref{flux-2}) in a ``quadrature'' form as follows 
\begin{equation}
\label{flux-3}
u'^2(r) = \frac{I^2(r)\varrho^{-2}(r) }{\varrho^2(r)\xi^{2(n-1)}(r)-I^2(r)},
\end{equation}
where
\[
I(r) = \int_0^r n\widetilde H(\tau) \varrho(\tau) \xi^{n-1}(\tau)\, {\rm d}\tau.
\]
For instance, in the case when $\widetilde H$ is constant we have to impose a condition such as
\begin{equation}
\label{ratio}
n|\widetilde H| \le \liminf_{r\to\infty} \frac{\varrho(r)\xi^{n-1}(r)}{\int_0^r \varrho(\tau) \xi^{n-1}(\tau)\, {\rm d}\tau}
\end{equation}
in order to guarantee the existence of radial solutions $u=u(r)$ to (\ref{PDE}) for model manifolds. Note that the right-hand side in (\ref{ratio}) is a sort of weighted isoperimetric ratio in $M$ with respect to the density $\varrho(r(x)) = |X(x)|$. 
By l'Hospital's rule we see that \eqref{ratio} is equivalent to the requirement
	\begin{equation}\label{asympHreq}
	n|\widetilde H| \le \liminf_{r\to\infty} \, (n-1) \frac{\xi'(r)}{\xi(r)} + \frac{\varrho'(r)}{\varrho(r)}.
	\end{equation}

This discussion motivates us to define in the general case a function of the form
	\begin{align}\label{u_+definition}
	u_+(x)  &= u_+\big(r(x)\big) \nonumber \\ 
	&= \bigintssss_{r(x)}^{+\infty} \frac{\int_0^\tau n\widetilde H(s) \varrho_+(s) \xi_+^{n-1}(s)\, {\rm d}s }{\varrho_+(\tau)\sqrt{\varrho_+^2(\tau)\xi_+^{2(n-1)}(\tau)-\big(\int_0^\tau n\widetilde H(s) \varrho_+(s) \xi_+^{n-1}(s)\, {\rm d}s\big)^2}}\, {\rm d}\tau \\
&\quad + ||\varphi||_{C^0(\partial_\infty M)} 
	\end{align} 
for some nonnegative functions $\varrho_+(r(x))$, $\xi_+(r(x))$ and $\widetilde H(r(x))$ to be chosen later. 

Plugging $u_+(x) = u_+(r(x))$ into the differential operator
\[
\mathcal{Q}[u] = {\rm div} \Big(\frac{\nabla u}{W}\Big) +\Big\langle \nabla \log\varrho, \frac{\nabla u}{W}\Big\rangle - nH
\]
 yields
\begin{eqnarray*}
& &  \mathcal{Q}[u_+] =\bigg\langle \nabla\frac{u_+'(r)}{(\varrho^{-2}(x)+{u'}^2_+(r))^{1/2}},  \partial_r\bigg\rangle \\
& & \qquad + \frac{u_+'(r)}{(\varrho^{-2}(x)+{u'}^2_+(r))^{1/2}} \bigg( {\rm div}\,\partial_r +\frac{1}{\varrho}\langle \nabla \varrho, \partial_r\rangle\bigg)-nH\\
& &  = \partial_r\bigg(\frac{u_+'(r)}{(\varrho^{-2}(x)+{u'}^2_+(r))^{1/2}}\bigg) + \frac{u_+'(r)}{(\varrho^{-2}(x)+{u'}^2_+(r))^{1/2}} \bigg( \Delta r +\frac{1}{\varrho}\langle \nabla \varrho, \partial_r\rangle\bigg)-nH\\
& & = \partial_r\bigg(\frac{u_+'(r)}{(\varrho^{-2}_+(r)+{u'}^2_+(r))^{1/2}}\frac{{(\varrho^{-2}_+(r)+{u'}^2_+(r))^{1/2}}}{(\varrho^{-2}(x)+{u'}^2_+(r))^{1/2}}\bigg)\\
& & \qquad + \frac{u_+'(r)}{(\varrho^{-2}(x)+{u'}^2_+(r))^{1/2}} \bigg( \Delta r +\frac{1}{\varrho}\langle \nabla \varrho, \partial_r\rangle\bigg)-nH\\
& & = \frac{{(\varrho^{-2}_+(r)+{u'}^2_+(r))^{1/2}}}{(\varrho^{-2}(x)+{u'}^2_+(r))^{1/2}} \bigg[\frac{u_+'(r)}{(\varrho^{-2}_+(r)+{u'}^2_+(r))^{1/2}}\bigg(\Delta r +\frac{1}{\varrho}\langle \nabla \varrho, \partial_r\rangle\bigg)\\
& & \qquad +\partial_r\bigg(\frac{u_+'(r)}{(\varrho^{-2}_+(r)+{u'}^2_+(r))^{1/2}} \bigg)\bigg]
\\
& &\qquad 
+ \frac{u_+'(r)}{(\varrho^{-2}_+(r)+{u'}^2_+(r))^{1/2}} \partial_r \bigg(\frac{{(\varrho^{-2}_+(r)+{u'}^2_+(r))^{1/2}}}{(\varrho^{-2}(x)+{u'}^2_+(r))^{1/2}}\bigg)-nH.
\end{eqnarray*}
Moreover, suppose that
	\begin{equation}\label{rho+assumption}
	\frac{\p_r\varrho(x)}{\varrho(x)} \ge 
	\frac{\varrho_+'\big(r(x)\big)}{\varrho_+\big(r(x)\big)}
	\end{equation}
for some 
positive and increasing $C^1$-function $\varrho_+\colon [0,\infty)\to (0,\infty)$ such that $\varrho_+(0)
=\varrho(o)$. By our choice of $u_+$, 
\[
u_+^\prime(r)= -\frac{\int_0^r n\widetilde H(s) \varrho_+(s) \xi_+^{n-1}(s)\, {\rm d}s }{\varrho_+(r)\sqrt{\varrho_+^2(r)\xi_+^{2(n-1)}(r)-\big(\int_0^r n\widetilde H(s) \varrho_+(s) \xi_+^{n-1}(s)\, {\rm d}s\big)^2}},
\]
and therefore
\begin{align*}
-n\widetilde H=\bigg(\frac{u'_+(r)}{(\varrho^{-2}_+(r)+u'^2_+(r))^{1/2}} \bigg)' + \frac{u'_+(r)}{(\varrho^{-2}_+(r)+u'^2_+(r))^{1/2}}  \bigg(\frac{\varrho'_+(r)}{\varrho_+(r)} + (n-1)\frac{\xi'_+(r)}{\xi_+(r)}\bigg).
\end{align*}
Hence we obtain 
\begin{align*}
& \mathcal{Q}[u_+] 
 = \frac{{(\varrho^{-2}_+(r)+{u'}^2_+(r))^{1/2}}}{(\varrho^{-2}(x)+{u'}^2_+(r))^{1/2}} \bigg[\frac{u_+'(r)}{(\varrho^{-2}_+(r)+{u'}^2_+(r))^{1/2}}\bigg(\Delta r +\frac{1}{\varrho}\langle \nabla \varrho, \partial_r\rangle\bigg)   \\ &\,\, +\partial_r\bigg(\frac{u_+'(r)}{(\varrho^{-2}_+(r)+{u'}^2_+(r))^{1/2}} \bigg)\bigg]
 +\frac{u_+'(r)}{(\varrho^{-2}_+(r)+{u'}^2_+(r))^{1/2}} \partial_r \bigg(\frac{{(\varrho^{-2}_+(r)+{u'}^2_+(r))^{1/2}}}{(\varrho^{-2}(x)+{u'}^2_+(r))^{1/2}}\bigg) - nH
\\
& \, \, \le - \frac{{(\varrho^{-2}_+(r)+{u'}^2_+(r))^{1/2}}}{(\varrho^{-2}(x)+{u'}^2_+(r))^{1/2}} 
n\widetilde H + 	\frac{u_+'(r)}{(\varrho^{-2}_+(r)+{u'}^2_+(r))^{1/2}} 
\partial_r \bigg(\frac{{(\varrho^{-2}_+(r)+{u'}^2_+(r))^{1/2}}}{(\varrho^{-2}(x)+{u'}^2_+(r))^{1/2}}\bigg) -nH.
\end{align*}
In order to prove that $u_+$ is indeed an upper barrier we next check that 
\begin{equation}
\partial_r \bigg(\frac{{(\varrho^{-2}_+(r)+{u'}^2_+(r))^{1/2}}}{(\varrho^{-2}(x)+{u'}^2_+(r))^{1/2}}\bigg)\ge 0.
\end{equation}
Note that $u_+'\le0$. We observe that
\[
\frac{\p}{\p r} \left(\sqrt{ \frac{\varrho_+(r)^{-2} + (u_+'(r))^2}{\varrho(x)^{-2} + (u_+'(r))^2} }\right)\ge 0
\]
if and only if
	\begin{equation}\label{derivativeestimate}
	(\varrho_+^{-2} + (u_+')^2)\Big(\frac{\p_r \varrho}{\varrho^{3}} - u_+' u_+''\Big) \ge
	(\varrho^{-2} +(u_+')^2)\Big(\frac{\varrho_+'}{\varrho_+^{3}} - u_+' u_+''\Big).
	\end{equation}
But now integrating \eqref{rho+assumption} we get 
	$$
	\log \varrho(x) \ge \log \varrho_+\big(r(x)\big)
	$$
which implies 
	\[
	\frac{1}{\varrho(x)} \le \frac{1}{\varrho_+\big(r(x)\big)}
	\]
and furthermore assuming 
	\[
	\frac{\p_r \varrho(x)}{\varrho(x)^3} \ge \frac{\varrho_+'\big(r(x)\big)}{\varrho_+\big(r(x)\big)^3}
	\]
we see that \eqref{derivativeestimate} holds.


Therefore we are left to show that 
	\[
- nH \le \sqrt{\frac{\varrho^{-2}_+(r)+{u'}^2_+(r)}{\varrho^{-2}(x)+{u'}^2_+(r)} }
n\widetilde H .
	\]
The conditions \eqref{ratio} and \eqref{asympHreq} in our mind, we choose $\widetilde H$ as
	\begin{equation}\label{Htildechoice}
		n\widetilde H(r) = (1-\ve) \left(\frac{\varrho_+'(r)}{\varrho_+(r)} + (n-1) \frac{\xi_+'(r)}{\xi_+(r)} \right) 
		\end{equation}
with some $\ve\in(0,1)$. Note that then
	\[
	\int_0^r n\widetilde H(s) \varrho_+(s) \xi_+^{n-1}(s) \, {\rm d}s = (1-\ve) \varrho_+(r) \xi_+^{n-1}(r)
	\]
and we see that with this choice the denominator in the definition of $u_+$ stays bounded from 0.
Moreover, we have 
	\[
	u_+'(r) = - \frac{1-\ve}{\varrho_+(r)\sqrt{2\ve-\ve^2}}
	\]
and therefore $u_+$ is well defined, positive and decreasing function if
	\begin{equation}\label{rhopluscond}
	\int_1^\infty \frac{1}{\varrho_+(r)} \, {\rm d}r <\infty.
	\end{equation}
Now we can compute
\begin{align*}
\frac{{\varrho^{-2}_+(r)+{u'}^2_+(r)}}{\varrho^{-2}(x)+{u'}^2_+(r)} &= 
 \frac{{\varrho^{-2}_+(r)+\big( (1-\ve)/(\varrho_+(r)\sqrt{2\ve-\ve^2})\big)^2}}{\varrho^{-2}(x)+\big( (1-\ve)/(\varrho_+(r)\sqrt{2\ve-\ve^2})\big)^2} \\
 &= \frac{\varrho^{-2}_+(r)\big(1+ (1-\ve)^2/(2\ve-\ve^2)\big)}{\varrho^{-2}(x)+\varrho_+^{-2}(r)(1-\ve^2)/(2\ve-\ve^2)},
\end{align*}
and for example, taking $\ve=1-\sqrt{2}/2$ we have
\[
\frac{{\varrho^{-2}_+(r)+{u'}^2_+(r)}}{\varrho^{-2}(x)+{u'}^2_+(r)} = \frac{2\varrho_+^{-2}(r)}{\varrho^{-2}(x)+\varrho_+^{-2}(r)}.
\]
For the prescribed mean curvature we obtain the bound
\[
-n H(x) \le (1-\ve) \sqrt{\frac{\varrho^{-2}_+(r)\big(1+ (1-\ve)^2/(2\ve-\ve^2)\big)}{\varrho^{-2}(x)+\varrho_+^{-2}(r)(1-\ve^2)/(2\ve-\ve^2)}} \left(\frac{\varrho_+'(r)}{\varrho_+(r)} + (n-1)\frac{\xi_+'(r)}{\xi_+(r)} \right)
\]
which implies that $\mathcal{Q}[u_+]\le 0$.
Similarly, $\mathcal{Q}[-u_+]\ge 0$ if 
\[
n H(x) \le (1-\ve) \sqrt{\frac{\varrho^{-2}_+(r)\big(1+ (1-\ve)^2/(2\ve-\ve^2)\big)}{\varrho^{-2}(x)+\varrho_+^{-2}(r)(1-\ve^2)/(2\ve-\ve^2)}} \left(\frac{\varrho_+'(r)}{\varrho_+(r)} + (n-1)\frac{\xi_+'(r)}{\xi_+(r)} \right).
\]

All together, we have obtained the following. 
\begin{lem}\label{unif_height_estim}
Let $M$ be a complete Riemannian manifold with a pole $o$ and consider the warped product manifold $M\times_\varrho \R$,
where $\varrho$ satisfies
\begin{equation}\label{rho_height_assum}
	\frac{\p_r\varrho(x)}{\varrho(x)} \ge \frac{\varrho_+'\big(r(x)\big)}{\varrho_+\big(r(x)\big)},
	\quad  \frac{\p_r \varrho(x)}{\varrho(x)^3} \ge \frac{\varrho_+'\big(r(x)\big)}{\varrho_+\big(r(x)\big)^3}
	\end{equation}
for some 
positive and increasing $C^1$-function $\varrho_+\colon [0,\infty)\to (0,\infty)$ such that 
\begin{equation}\label{rho_height_assum1}
\varrho_+(0) =\varrho(o)\ \text{ and }\ 
\int_1^\infty \varrho_+(s)^{-1}\, {\rm d}s<\infty.
\end{equation}
 Furthermore, assume that the radial sectional curvatures of $M$ are
bounded from above by
	\[
	K_M(P_x) \le - \frac{\xi_+''\big(r(x)\big)}{\xi_+\big(r(x)\big)}
	\]
and that the prescribed mean curvature function satisfies
	\begin{align}\label{height_mean_assum}
	&n |H(x)| \le \\ 
	&(1-\ve) \sqrt{\frac{\varrho^{-2}_+\big(r(x)\big)\big(1+ (1-\ve)^2/(2\ve-\ve^2)\big)}{\varrho^{-2}(x)
	+\varrho_+^{-2}\big(r(x)\big)(1-\ve^2)/(2\ve-\ve^2)}} \left(\frac{\varrho_+'\big((r)\big)}{\varrho_+\big(r(x)\big)} 
	+ (n-1)\frac{\xi_+'\big(r(x)\big)}{\xi_+\big(r(x)\big)} \right) \nonumber
	\end{align}
	for some $\ve\in (0,1)$.
Then the function $u_+$ defined by \eqref{u_+definition} and \eqref{Htildechoice} satisfies $\mathcal{Q}[u_+] \le 0$
and $u_+ \ge ||\varphi||_{C^0}$ in $M$ with 
	\begin{equation}\label{u+_infty}
	u_+(r) \to ||\varphi||_{C^0} \quad \text{as } r\to\infty.
	\end{equation}
Furthermore $\mathcal{Q}[-u_+]\ge0$ and $-u_+ \le -||\varphi||_{C^0}$ in $M$.
\end{lem}
\begin{rem}
In particular, if the sectional curvatures of a Cartan-Hadamard manifold $M$ are bounded from above as
\begin{equation}\label{Kleq-a}
K_M(P_x)\le - a\big(r(x)\big)^2
\end{equation}
for some smooth function $a\colon [0,\infty)\to [0,\infty)$, the condition \eqref{height_mean_assum} reads as
\begin{align}\label{height_mean_assum2}
	&n |H(x)| \le \\ 
	&(1-\ve) \sqrt{\frac{\varrho^{-2}_+\big(r(x)\big)\big(1+ (1-\ve)^2/(2\ve-\ve^2)\big)}{\varrho^{-2}(x)
	+\varrho_+^{-2}\big(r(x)\big)(1-\ve^2)/(2\ve-\ve^2)}} \left(\frac{\varrho_+'\big((r)\big)}{\varrho_+\big(r(x)\big)} 
	+ (n-1)\frac{f_a'\big(r(x)\big)}{f_a\big(r(x)\big)} \right), \nonumber
	\end{align}
	with $f_a$ as in \eqref{jac-equ}.
	\end{rem}
In a rotationally symmetric case if $\varrho=\varrho_+(r)$ (and \eqref{rhopluscond} holds), we see that the bound for the 
mean curvature is 
\begin{align*}
	n |H(x)| \le(1-\ve)  
	\left(\frac{\varrho_+'\big(r(x)\big)}{\varrho_+\big(r(x)\big)} 
	+ (n-1)\frac{\xi_+'\big((r(x\big)}{\xi_+\big(r(x)\big)} \right).
	\end{align*}

\subsection{Example: hyperbolic space}

We consider the warped model of $\mathbb{H}^{n+1}$ given by $\mathbb{H}^n\times_{\cosh r} \mathbb{R}$, where $r$ is a radial coordinate in $\mathbb{H}^n$ defined with respect to a fixed reference point $o\in\mathbb{H}^n$. Then the hyperbolic metric is expressed as
\[
\cosh^2 {\rm d}t^2 + {\rm d}r^2 + \sinh^2 r \, {\rm d}\vartheta^2,
\]
where ${\rm d}\vartheta^2$ stands for the standard metric in $\mathbb{S}^{n-1}\subset T_o \mathbb{H}^n$. The flow of the Killing field $X=\partial_t$ is given by the hyperbolic translations generated by a geodesic $\gamma$ orthogonal to $\mathbb{H}^n$ through $o$. Since $\varrho(r) = \cosh r$ and $\xi(r) = \sinh r$ in this case, we obtain
\begin{eqnarray*}
& & \lim_{r\to\infty}\frac{\varrho(r) \xi^{n-1}(r)}{\int_0^r \varrho(\tau) \xi^{n-1}(\tau) \, {\rm d}\tau} = \lim_{r\to\infty}\frac{\sinh^n  r  + (n-1)\cosh^2 r \sinh^{n-2}r }{\cosh r \sinh^{n-1}r} \\
& &\,\, = \lim_{r\to\infty}\left(\frac{\sinh r}{\cosh r} + (n-1)\frac{\cosh r}{\sinh r} \right)\ge n.
\end{eqnarray*}
Therefore a natural bound to the mean curvature function according (\ref{ratio}) is 
\[
|H|<1,
\]
that is, below the mean curvature of horospheres. 

We also have for $|H|<1$
\begin{eqnarray*}
& & \frac{I^2(r)\varrho^{-2}(r) }{\varrho^2(r)\xi^{2(n-1)}(r)-I^2(r)} \le \frac{\sinh^{2n} r\cosh^{-2}(r) }{\cosh^2 r \sinh^{2(n-1)} r-\sinh^{2n} r}= \frac{\sinh^2 r}{\cosh^2 r}.
\end{eqnarray*}
Therefore we have
\[
u'^2 (r) \le 1.
\]
If $|H|={\rm cte.}<1$ we have an explicit expression
\[
u'^2(r)  = \frac{H^2}{\cosh^2 r - H^2 \sinh^2 r} \frac{\cosh^2 r}{\sinh^2 r}.
\]


\subsection{Global barrier $V$}
In this subsection we construct a global barrier using an idea of Mastrolia, Monticelli, and Punzo \cite{MMP}; see also \cite{CHH2}.
Recall that
$\varrho_+\colon [0,\infty)\to (0,\infty)$ is an increasing smooth 
function satisfying $\varrho_+(0)=\varrho(o)$ and
\begin{equation}\label{rho+assumption2}
	\frac{\p_r\varrho(x)}{\varrho(x)} \ge \frac{\varrho_+'\big(r(x)\big)}{\varrho_+\big(r(x)\big)}
\end{equation} for all $x\in M$.
 Then we have an estimate
	\begin{equation}\label{laplacecomparison2}
	\Delta_{-\log \varrho} r (x) \ge (n-1) \frac{f_a'(r(x))}{f_a(r(x))} + \frac{\varrho_+'(r(x))}{\varrho_+(r(x))}
	\end{equation}
	for the weighted Laplacian of the distance function $r$.
Let $a_0$ be a positive function such that 
\begin{equation}\label{a0cond}
\int_0^{\infty} \left( \int_t^\infty \frac{{\rm d}s}{\varrho_+^2(s)f_a^{n-1}(s)}\right) a_0(t) 
 f_a^{n-1}(t){\rm d}t <\infty.
 \end{equation}
We define 
  \begin{equation} \label{Vdefinition}\begin{split}
      V(x)&= \left(\int_{r (x)}^\infty  \frac{{\rm d}s}{\varrho_+^2(s)f_a^{n-1}(s)}\right) \left(\int_0^{r(x)} a_0(t)
      f_a^{n-1}(t){\rm d}t \right) \\ &\quad - \int_0^{r(x)} \left( \int_t^\infty  \frac{{\rm d}s}{\varrho_+^2(s) f_a^{n-1}(s)}\right) a_0(t) 
  f_a^{n-1}(t){\rm d}t - D + ||\varphi||_\infty,\end{split}
    \end{equation}
 where $D$ is the constant given by \eqref{Ddefi}.
Denoting $V(r)=V(r(x))$, we observe that 
  \begin{align*}
      V'(r) &= - \frac{1}{\varrho_+^2(r)f_a^{n-1}(r)} \int_0^r a_0(t)   f_a^{n-1}(t){\rm d}t < 0\\
      \noalign{and}  
      V''(r)&=\frac{1}{\varrho_+^2(r)f_a^{n-1}(r)}\left(\frac{(n-1)f_a^\prime(r)}{f_a (r)} + \frac{2\varrho_+^\prime(r)}{\varrho_+(r)}\right) \int_0^r a_0(t)f_a^{n-1}(t){\rm d}t   -\frac{a_0(r)}{\varrho_+^2(r)}.
       \end{align*}
Since $V'(r)<0$, the limit
  \begin{equation}\label{Ddefi}\begin{split}
	D &=\lim_{r\to\infty} \Big\{ \int_{r}^\infty \frac{{\rm d}s}{\varrho_+^2(s)f_a^{n-1}(s)} \int_0^{r} a_0(t)
      f_a^{n-1}(t){\rm d}t\\
      &\quad  - \int_0^{r}  \int_t^\infty \frac{{\rm d}s}{\varrho_+^2(s)f_a^{n-1}(s)} a_0(t)
      f_a^{n-1}(t){\rm d}t \Big\} \end{split}
      \end{equation}      
      exists. Furthermore, $D\le 0$ (see \cite[(4.5)]{MMP}) and finite by \eqref{a0cond} and therefore $V$ is well defined. 
Next we write 
\begin{align}\label{Qfraction}
     &\mathcal{Q}[V]\\
     & = \frac{(\varrho^{-2}+|\nabla V|^2) \Delta_{-\log \varrho} V - (\varrho^{-2}+|\nabla V|^2)^{3/2}nH(x) - \frac{1}{2}
  \ang{\nabla (\varrho^{-2} +|\nabla V|^2), \nabla V}}{(\varrho^{-2}+|\nabla V|^2)^{3/2}}\nonumber
   \end{align}
   and aim to prove that $\mathcal{Q}[V]\le 0$. First we estimate the weighted
 Laplacian of $V$ by using \eqref{laplacecomparison2}
\begin{align*}
     \Delta_{-\log \varrho} V &= V''(r) +V'(r)  \Delta_{-\log \varrho} r \\ 
  &   \le V''(r) +\left( (n-1)\frac{f_a^\prime(r)}{f_a(r)} + \dfrac{\varrho_+'(r)}{\varrho_+(r) } \right)V'(r) \\
      &=\frac{1 }{\varrho_+^2(r)f_a^{n-1}(r)}\left(\frac{(n-1)f_a^\prime(r)}{f_a (r)} +\frac{2\varrho_+^\prime(r)}{\varrho_+(r)}\right)  \int_0^r a_0(t) f_a^{n-1}(t){\rm d}t\\
& - \frac{a_0(r)}{\varrho_+^2(r)} 
       -  \frac{1}{\varrho_+^2(r)f_a^{n-1}(r)}\left(\frac{(n-1)f_a^\prime(r)}{f_a (r)} +\frac{\varrho_+^\prime (r)}{\varrho_+(r)}  \right) 
      \int_0^r a_0(t) f_a^{n-1}(t){\rm d}t \\
      &= - \frac{a_0(r)}{\varrho_+^2(r)}+\frac{\varrho_+^\prime(r)}
      {\varrho_+^3(r)f_a^{n-1}(r)}\int_0^r a_0(t) f_a^{n-1}(t){\rm d}t\\
   &= - \frac{a_0(r)}{\varrho_+^2(r)}-\frac{\varrho_+^\prime(r)}{\varrho_+(r)}V^\prime (r), 
    \end{align*}
  and thus the first term of \eqref{Qfraction} can be estimated as
\[
     \big(\varrho^{-2} + |\nabla V|^2\big) \Delta_{-\log \varrho} V 
     \le -\left(\varrho^{-2} + (V^\prime(r))^2\right) 
 \left( \frac{a_0(r)}{\varrho_+^2(r)}+\frac{\varrho_+^\prime(r)}
      {\varrho_+(r)}V^\prime (r)
  \right).    
\]     
Then, for the last term of \eqref{Qfraction} we have 
 \begin{align*}
    -\frac{1}{2}&\ang{\nabla (\varrho^{-2}+|\nabla V|^2),\nabla V} 
  = -\big(V'(r)\big)^2 V''(r) +\frac{\partial_r\varrho}{\varrho^3} V'(r)\\
  &=- (V'(r))^2 \left(\left(\frac{(n-1)f_a^\prime(r)}{f_a (r)} + \frac{2\varrho_+^\prime(r)}{\varrho_+(r)}\right)V^\prime (r)   -\frac{a_0(r)}{\varrho_+^2(r)}\right)
+ \frac{\partial_r\varrho}{\varrho^3} V^\prime (r).
   \end{align*}
   Hence
   \begin{align*}
  \varrho_+^2(r)&\big(\varrho^{-2} + |\nabla V|^2\big) \Delta_{-\log \varrho} V     -\frac{1}{2}\varrho_+^2(r)\ang{\nabla (\varrho^{-2}+|\nabla V|^2),\nabla V} \\
&\le -\varrho^{-2}a_0(r)-\varrho^{-2}\varrho_+^2(r)\left(\frac{\varrho_+^\prime (r)}{\varrho_+(r)}-\frac{\partial_r\varrho}{\varrho}\right)V^\prime (r)\\
&\quad -\varrho_+^2(r)\big(V^\prime (r)\big)^3\left(\frac{(n-1)f_a^\prime(r)}{f_a(r)} +\frac{\varrho_+^\prime(r)}{\varrho_+(r)}\right)\\
&\le  -\varrho^{-2}a_0(r) -\varrho_+^2(r)\big(V^\prime (r)\big)^3 \left(\frac{(n-1)f_a^\prime(r)}{f_a(r)} +\frac{\varrho_+^\prime(r)}{\varrho_+(r)}\right).
   \end{align*}
Finally, if the prescribed mean curvature function satisfies
\[
-nH \le \frac{\varrho^{-2}\varrho_+^{-2}(r)a_0(r) 
+ \big(-V^\prime(r)\big)^3\left(\frac{(n-1)f_a^\prime(r)}{f_a (r)} + \frac{\varrho_+^\prime(r)}{\varrho_+(r)}\right)}{\left(\varrho^{-2} +
\big(V^\prime(r)\big)^2\right)^{3/2} }
\]
in $M$, we obtain $\mathcal{Q}[V]\le 0$ as desired. Similarly, we see that $\mathcal{Q}[-V]\ge 0$ if 
\[
nH \le \frac{\varrho^{-2}\varrho_+^{-2}(r)a_0(r) 
+ \big(-V^\prime(r)\big)^3\left(\frac{(n-1)f_a^\prime(r)}{f_a (r)} + \frac{\varrho_+^\prime(r)}{\varrho_+(r)}\right)}{\left(\varrho^{-2} +
\big(V^\prime(r)\big)^2\right)^{3/2} }.
\]

Hence we have proved the following uniform height estimate.
\begin{lem}\label{unif-heigh2}
Let $\varphi\colon M\to \R$ be a bounded function and assume that the prescribed mean curvature function $H$ and the function $V$ defined in  \eqref{Vdefinition} satisfy
\begin{equation}\label{HVbound}
n|H| \le \frac{\varrho^{-2}\varrho_+^{-2}(r)a_0(r) 
+ \big(-V^\prime(r)\big)^3\left(\frac{(n-1)f_a^\prime(r)}{f_a (r)} + \frac{\varrho_+^\prime(r)}{\varrho_+(r)}\right)}{\left(\varrho^{-2} +
\big(V^\prime(r)\big)^2\right)^{3/2}},
\end{equation}
with some positive functions $\varrho_+$ and $a_0$ satifying 
\eqref{rho+assumption2} and \eqref{a0cond}, respectively. Then 
\begin{equation}\label{Vsuper}
	\mathcal{Q}[V] = \dv_{-\log\varrho} \frac{\nabla V}{\sqrt{\varrho^{-2} + |\nabla V|^2}} -nH \le 0\quad\text{in }M,
      \end{equation}
 \begin{equation}\label{Vheight}
       V(x) > ||\varphi||_\infty \quad \text{for all } x\in M,
      \end{equation}
and
      \begin{equation}\label{Vlimes}
       \lim_{r(x)\to\infty} V(x) = ||\varphi||_\infty.
      \end{equation}
Furthermore, $\mathcal{Q}[-V] \ge 0$ in $M$. 
\end{lem}

Next we discuss possible choices of the functions $\varrho_+$ and $a_0$ and their influence on the bound of $|H|$.
Notice that the right hand side of \eqref{HVbound} can be written as
\begin{equation}\label{HVound2}
\dfrac{\frac{\varrho\varrho_+^{-2}(r)a_0(r)}{\big(-V^\prime(r)\varrho\big)^3}+\frac{(n-1)f_a^\prime(r)}{f_a (r)} + \frac{\varrho_+^\prime(r)}{\varrho_+(r)}}{\big(1+\big(-V^\prime(r)\varrho\big)^{-2}\big)^{3/2}}.
\end{equation}
Hence if we can choose the comparison manifold 
$M_{-a^2(r)}\times_{\varrho_+} \R$ and $a_0$ such that 
$V^\prime(r)\varrho\to -\infty$ and
\[
\frac{\varrho\varrho_+^{-2}(r)a_0(r)}{\big(-V^\prime(r)\varrho\big)^3}
\to 0
\]
as $r\to\infty$, we obtain 
\begin{equation}\label{asbeh}
n|H|\le \frac{(n-1)f_a^\prime(r)}{f_a (r)} + \frac{\varrho_+^\prime(r)}{\varrho_+(r)}
\end{equation}
asymptotically as $r\to\infty$.

\begin{exa}
In the hyperbolic case $\HH^{n+1}=\HH^n \times_{\cosh r}\R$ we may take $\varrho_+(r)=\varrho=\cosh$. Choosing $a_0(r)=\sinh^\alpha r$ for some $\alpha\in (1,2)$ yields to the natural asymptotic bound 
$|H|<1$ as $r\to\infty$.
\end{exa}
\begin{exa}
More generally, if $N=M\times_\varrho\R$, where the sectional curvatures of $M$ have a negative upper bound $-k^2$ and if the warping function $\varrho$ satisfies \eqref{rho+assumption2} with 
$\varrho_+(r)\ge c_1e^{\alpha r}$ for some $\alpha>0$, then $f_a(r)\approx e^{kr}$ and  \eqref{a0cond} holds if
\[
\int_0^\infty a_0(t)e^{-2\alpha t}{\rm d}t<\infty.
\]
Moreover, if $\varrho_+(r)\le c_2e^{\beta r}$ for some $0<\beta<2\alpha$, then by choosing $a_0(t)=e^{\kappa t},\ \beta<\kappa<2\alpha$, we get
\eqref{asbeh} asymptotically as $r\to\infty$.
\end{exa}
\begin{exa}
If $N=M\times_\varrho\R$, where the sectional curvatures of $M$ have a negative upper bound 
\[
K(P_x)\le -\frac{\phi(\phi-1)}{r(x)^2},\quad \phi>1,
\]
and if the warping function $\varrho$ satisfies \eqref{rho+assumption2} with 
$\varrho_+(r)= c r^{\alpha},\ \alpha>1$, then $f_a(r)\approx r^{\phi}$ and  \eqref{a0cond} holds if
\[
\int_0^\infty a_0(r)r^{-2\alpha+1}{\rm d}r<\infty.
\]
Choosing $a_0(r)=r^\kappa$, for some $\alpha-1<\kappa<2(\alpha-1)$, we get \eqref{asbeh} asymptotically as $r\to\infty$.
\end{exa}

\section{Barrier at infinity}\label{sec_bar_infty}

In this section we assume that $M$ is a Cartan-Hadamard manifold of dimension $n\ge 2$, $\partial_{\infty}M$ is the asymptotic boundary of $M$, 
and $\bar{M}=M\cup\partial_{\infty}M$ the compactification of $M$ in the cone topology.
Recall that the asymptotic boundary is defined as the set of all equivalence classes of unit speed
geodesic rays in $M$; two such rays $\gamma_{1}$ and $\gamma_{2}$ are
equivalent if $\sup_{t\ge0}d\bigl(\gamma_{1}(t),\gamma_{2}(t)\bigr)< \infty$. The equivalence class of $\gamma$ is denoted by $\gamma(\infty)$.
For each $x\in M$ and $y\in\bar{M}\setminus\{x\}$ there exists a unique unit
speed geodesic $\gamma^{x,y}\colon\mathbb{R}\to M$ such that $\gamma^{x,y}%
_{0}=x$ and $\gamma^{x,y}_{t}=y$ for some $t\in(0,\infty]$. If $v\in
T_{x}M\setminus\{0\}$, $\alpha>0$, and $r>0$, we define a cone
\[
C(v,\alpha)=\{y\in\bar M\setminus\{x\}:\sphericalangle(v,\dot\gamma^{x,y}
_{0})<\alpha\}
\]
and a truncated cone
\[
T(v,\alpha,r)=C(v,\alpha)\setminus\bar B(x,r),
\]
where $\sphericalangle(v,\dot\gamma^{x,y}_{0})$ is the angle between vectors
$v$ and $\dot\gamma^{x,y}_{0}$ in $T_{x} M$. All cones and open balls in $M$
form a basis for the cone topology on $\bar M$.

Throughout this section, we assume that the sectional curvatures of $M$ are bounded from below and above
by
    \begin{equation}
    \label{curv-bound-gen}
     -(b\circ r)^2(x) \le K(P_x) \le -(a\circ r)^2 (x)
    \end{equation}
for all $x\in M$, 
where $r (x) = d(o,x)$ is the distance to a fixed point $o\in M$ and $P_x$ is any 2-dimensional subspace of $T_xM$. The functions
$a,b\colon [0,\infty) \to [0,\infty)$ are assumed to be smooth such that $a(t)=0$ and $b(t)$ is constant for $t\in [0,T_0]$ for some $T_0>0$, and that assumptions \eqref{A1}--\eqref{A7} hold. 
These curvature bounds are needed to control the first two derivatives of ``barrier'' functions that we will construct in the next subsection. 
We assume that function $b$ in \eqref{curv-bound-gen} is monotonic and that there exist positive constants 
$T_1\ge T_0, C_1, C_2, C_3$, and $Q\in (0,1)$ such that 
  \begin{align}
    \tag{A1}\label{A1}
    a(t)\begin{cases}=C_1t^{-1}&\text{if $b$ is decreasing,}\\ 
    \ge C_1t^{-1}&\text{if $b$ is increasing}\\ \end{cases} 
  \end{align}
for all $t\ge T_1$ and
  \begin{align}
    \tag{A2}\label{A2}
    a(t)&\le C_2, \\
    \tag{A3}\label{A3}
    b(t+1)&\le C_2b(t), \\
    \tag{A4}\label{A4}
    b(t/2)&\le C_2b(t), \\
    \tag{A5}\label{A5}
    b(t)&\ge C_3(1+t)^{-Q}
  \end{align}
for all $t\ge 0$.
In addition, we assume that 
  \begin{align}
    \tag{A6}\label{A6}
    &\lim_{t\to\infty}\frac{b'(t)}{b(t)^2}=0 
  \end{align}
and that there exists a constant $C_4>0$ such that
  \begin{align}
    \tag{A7}\label{A7}
    &\lim_{t\to\infty}\frac{t^{1+C_4}b(t)}{f_a'(t)}=0;
  \end{align}
  see \eqref{jac-equ} for the definition of $f_a$.

We recall from \cite{HoVa} the following two examples of functions $a$ and $b$.
\begin{exa}
\label{ex1} 
Let $C_{1}=\sqrt{\phi(\phi-1)},$ where $\phi>1$ is a constant. For $t\ge R_{0}$ let
\[
a(t)=\frac{C_{1}}{t}
\]
and
\[
b(t)=t^{\phi-2-\varepsilon/2},
\]
where $0<\varepsilon<2\phi-2$, 
and extend them to smooth functions $a\colon[0,\infty)\to (0,\infty)$ and $b\colon
[0,\infty)\to(0,\infty)$ such that they are constants in some
neighborhood of $0$, $b$ is monotonic and $b\ge a$. Then $a$ and $b$ satisfy \eqref{A1}-\eqref{A7} with constants
$T_{1}=R_{0}$, $C_{1}$, some $C_{2}>0$, some $C_{3}>0$, $Q=\max\{1/2,-\phi
+2+\varepsilon/2\}$, and any $C_{4}\in(0,\varepsilon/2)$.
It is
easy to verify that then
\[
f_{a}(t)=c_{1}t^{\phi}+c_{2}t^{1-\phi}
\]
for all $t\ge R_{0}$, where
\[
c_{1}=R_{0}^{-\phi}\frac{f_{a}(R_{0})(\phi-1)+R_{0}f_{a}^{\prime
}(R_{0})}{2\phi-1}>0,
\]
and
\[
c_{2}=R_{0}^{\phi-1}\frac{f_{a}(R_{0})\phi-R_{0}f_{a}^{\prime}(R_{0}%
)}{2\phi-1}.
\]
We then have
\[
\lim_{t\to\infty}\frac{tf_{a}^{\prime}(t)}{f_{a}(t)}=\phi
\]
and, for all $C_{4}\in(0,\varepsilon/2)$
\[
\lim_{t\to\infty}\frac{t^{1+C_{4}}b(t)}{f_{a}^{\prime}(t)}=0.
\]
It follows that $a$ and $b$ satisfy \eqref{A1}-\eqref{A7} with constants
$T_{1}=R_{0}$, $C_{1}$, some $C_{2}>0$, some $C_{3}>0$, $Q=\max\{1/2,-\phi
+2+\varepsilon/2\}$, and any $C_{4}\in(0,\varepsilon/2)$.
\end{exa}

\begin{exa}
\label{ex2} Let $k>0$ and $\varepsilon>0$ be constants and define $a(t)=k$ for
all $t\ge0$. Define
\[
b(t)=t^{-1-\varepsilon/2}e^{kt}
\]
for $t\ge R_{0}=r_{0}+1$, where $r_{0}>0$ is so large that $t\mapsto
t^{-1-\varepsilon/2}e^{kt}$ is increasing and greater than $k$ for all $t\ge
r_{0}$. Extend $b$ to an increasing smooth function $b\colon[0,\infty
)\to[k,\infty)$ that is constant in some neighborhood of $0$. 
We can
choose $C_{1}>0$ in \eqref{A1} as large as we wish. 
Then $a$ and $b$ satisfy
\eqref{A1}-\eqref{A7} with constants
 $C_{1},\ T_{1}=C_{1}/k$, some $C_{2}>0$,
some $C_{3}>0$, $Q=1/2$, and any $C_{4}\in(0,\varepsilon/2)$.
\end{exa}

\subsection{Construction of a barrier}\label{subsec_barrier_constr}

Following \cite{HoVa}, we construct a barrier function for each boundary point $x_0\in\partial_{\infty}M$. 
Towards this end let $v_{0}=\dot\gamma^{o,x_{0}}_{0}$ be the initial (unit) vector of
the geodesic ray $\gamma^{o,x_{0}}$ from a fixed point $o\in M$ and define a
function $h:\partial_{\infty}M\to\mathbb{R}$,
\begin{equation}
\label{eq:hoodef}h(x)=\min\bigl(1,L\sphericalangle(v_{0},\dot\gamma^{o,x}%
_{0})\bigr),
\end{equation}
where $L\in(8/\pi,\infty)$ is a constant. Then we define a crude extension 
$\tilde h\in C(\bar{M})$, with $\tilde h|\partial_{\infty}M=h$, by setting
\begin{equation}
\label{eq:hoodeftilde}\tilde h(x)=\min\Bigl(1,\max\bigl(2-2r
(x),L\sphericalangle(v_{0},\dot\gamma^{o,x}_{0})\bigr)\Bigr).
\end{equation}
Finally, we smooth out $\tilde{h}$ to get an extension
$h\in C^{\infty}(M)\cap C(\bar{M})$ with controlled first and second order 
derivatives. For that purpose, we fix $\chi\in C^{\infty}(\R)$ such that 
$0\le\chi\le 1$, $\supp\chi\subset[-2,2]$, and
$\chi\vert[-1,1]\equiv1$. Then for any function $\varphi\in C(M)$ we define
functions $F_{\varphi}\colon M\times M\to\mathbb{R},\ {\mathcal{R}}%
(\varphi)\colon M\to M$, and ${\mathcal{P}}(\varphi)\colon M\to\mathbb{R}$ by
\begin{align*}
F_{\varphi}(x,y)  &  =\chi\bigl(b(r(y))d(x,y)\bigr)\varphi(y),\\
{\mathcal{R}}(\varphi)(x)  &  =\int_{M}F_{\varphi}(x,y){\rm d}m(y),\ \text{ and}\\
{\mathcal{P}}(\varphi)  &  =\frac{{\mathcal{R}}(\varphi)}{{\mathcal{R}}(1)},
\end{align*}
where
\[
{\mathcal{R}}(1)(x)=\int_{M}\chi\bigl(b(r(y))d(x,y)\bigr){\rm d}m(y)>0.
\]
Thus ${\mathcal{P}}(\varphi)$ is an integral average of $\varphi$ with respect to $\chi$ similar to that in
\cite[p. 436]{andschoen} except that here the function $b$ is taken into account explicitly. 
If $\varphi\in C(\bar M)$, we extend ${\mathcal{P}}(\varphi)\colon
M\to\mathbb{R}$ to a function $\bar{M}\to\mathbb{R}$ by setting ${\mathcal{P}%
}(\varphi)(x)=\varphi(x)$ whenever $x\in M(\infty)$. Then the extended
function ${\mathcal{P}}(\varphi)$ is $C^{\infty}$-smooth in $M$ and continuous
in $\bar{M}$; see \cite[Lemma 3.13]{HoVa}. In particular, applying
${\mathcal{P}}$ to the function $\tilde{h}$ yields an appropriate smooth
extension
\begin{equation}
\label{extend_h}
h:={\mathcal{P}}(\tilde{h})
\end{equation}
of the original function $h\in C\bigl(\partial_{\infty}M\bigr)$ that was defined in \eqref{eq:hoodef}.

We denote
\[
\Omega=C(v_{0},1/L)\cap M \ \text{ and }\ \ell\Omega=C(v_{0},\ell/L)\cap M
\]
for $\ell>0$. We collect together all these constants and functions and denote
\[
C=(a,b,T_{1},C_{1},C_{2},C_{3},C_{4},Q,n,L).
\]
Furthermore, we denote by $\|\Hess_{x} u\|$ the norm of the Hessian of a
smooth function $u$ at $x$, that is
\[
\|\Hess_{x} u\|=\sup_{ \overset{ \mbox{\scriptsize$X\in T_xM$} }{\lvert X
\rvert\le1}}\lvert\Hess u(X,X) \rvert.
\]
The following lemma gives the desired estimates for derivatives of $h$.
We refer to \cite{HoVa} for the proofs of these estimates; see also \cite{CHR2}.

\begin{lem}\cite[Lemma 3.16]{HoVa}\label{arvio_lause}
There exist constants $R_1=R_1(C)$ and $c_1=c_1(C)$ such that the extended function 
$h\in C^\infty(M)\cap C(\bar M)$ in \eqref{extend_h}
satisfies 
\begin{equation}\label{arvio1}
  \begin{split}
  |\nabla h(x)|&\le c_1\frac{1}{(f_a\circ r)(x)}, \\
  \|\Hess_x h\|&\le c_1\frac{(b\circ r)(x)}{(f_a\circ r)(x)}, \\
  \end{split}
  \end{equation}
for all $x\in 3\Omega\setminus B(o,R_1)$. 
In addition,
\[h(x)=1
  \]
for every $x\in M\setminus\bigl(2\Omega\cup B(o,R_1)\bigr)$.
\end{lem}

Let $A>0$ be a fixed constant, and $R_3>0$ and $\delta>0$ constants that will be determined 
later, and $h$ the function defined in \eqref{extend_h}. We will show that a function
	\begin{equation}\label{psi_barr_def}
	\psi = A(R_3^\delta r^{-\delta} + h)
	\end{equation}
is a supersolution 
	\begin{align*}
	\mathcal{Q}[\psi] &= \dv_{-\log\varrho} \frac{\nabla \psi}{\sqrt{\varrho^{-2} + |\nabla \psi|^2}} -nH \\
	&= \dv \frac{\nabla \psi}{W} + \ang{\nabla \log \varrho,\frac{\nabla \psi}{W}}-nH < 0
	\end{align*}
in the $3\Omega \setminus \bar B(o,R_3)$. In the proof we shall use the following estimates obtained in 
\cite{HoVa}:
\begin{lem}\cite[Lemma 3.17]{HoVa}\label{perusta}
There exist constants $R_2=R_2(C)$ and $c_2=c_2(C)$ with the following property.
If $\delta\in(0,1)$, then
\[\begin{split}
  |\nabla h|&\le c_2/(f_a\circ r), \\
  \|\Hess h\|&\le c_2 r^{-C_4-1}(f_a'\circ r)/(f_a\circ r), \\
  |\nabla\langle\nabla h,\nabla h\rangle|&\le c_2 r^{-C_4-2}(f_a'\circ r)/(f_a\circ r), \\
  |\nabla\langle\nabla h,\nabla( r^{-\delta})\rangle|&\le c_2 r^{-C_4-2}(f_a'\circ r)/(f_a\circ r), \\
  \nabla\bigl\langle\nabla(r^{-\delta}),\nabla( r^{-\delta})\bigr\rangle
  &=-2\delta^2(\delta+1)r^{-2\delta-3}\nabla r
  \end{split}\]
in the set $3\Omega\setminus B(o,R_2)$.
\end{lem}

Let us  denote
\[
\phi=\frac{1+\sqrt{1+4C_{1}^{2}}}{2}>1,\quad\text{and}\quad\delta_{1}
=\min\left\lbrace C_{4}/2,\frac{-1+(n-1)\phi}{1+(n-1)\phi}\right\rbrace
\in(0,1),
\]
where $C_{1}$ and $C_{4}$ are constants defined in \eqref{A1} and \eqref{A7},
respectively. 

\begin{lem}\label{psi_barrier_lemma}
Assume that the prescribed mean curvature function $H$ satisfies 
	\begin{equation}\label{asym_meancurv_assum}
	\sup_{r(x)=t}n|H(x)|  <  \dfrac{C_0 t^{-\delta_1 -1}}{\sqrt{\varrho^{-2}(t)+( C_0 t^{-\delta-1})^2}}\left(  (n-1)\frac{f_a'(t)}{f_a(t)} + \frac{\p_r\varrho}{\varrho} - \dfrac{1}{t}  \right)
	\end{equation}
	for some positive constants  $C_0>1$ and $\delta<\min\{\delta_1,\phi-1\}$,  
and that the warping function $\varrho$ satisfies 
\begin{equation}\label{rho_sign_assum}
\max\left(0,-\frac{r\p_r\varrho}{\varrho}\right)
=o\left( \frac{rf_a'(r)}{f_a(r)}\right)
	\end{equation}
	and
	\begin{equation}\label{rho_rad_assum}
	|\nabla \varrho| =o\left( \frac{f_a(r)}{r^{\delta+1}}|\p_r\varrho|\right)
	\end{equation}
	as $r\to\infty$.
Then there exists a constant $R_3=R_3(C,C_0,\delta)\ge R_2$ such that the function $\psi$ defined in \eqref{psi_barr_def} satisfies $\mathcal{Q}[\psi]<0$ in the set
$3\Omega\setminus\bar B(o,R_3)$.
\end{lem}
\begin{proof}
In the proof we will denote by $c$ those positive constants whose actual value is irrelevant and may vary even within a line. 
Furthermore, the estimates will be done in $3\Omega\setminus \bar B(o,R_3)$, with $R_3$ large enough.
Note that 
	\begin{align*}
	\mathcal{Q}[\psi] &= \frac{\Delta_{-\log\varrho}\psi}{\sqrt{\varrho^{-2}+|\nabla\psi|^2}}- \frac{1}{2} \frac{\ang{\nabla(\varrho^{-2}+|\nabla\psi|^2),\nabla \psi}}{(\varrho^{-2}+|\nabla\psi|^2)^{3/2}} - nH \\
	&= \frac{(\varrho^{-2}+|\nabla\psi|^2)\Delta_{-\log\varrho}\psi - \frac{1}{2} \ang{\nabla(\varrho^{-2}+|\nabla\psi|^2),\nabla\psi} -  (\varrho^{-2}+|\nabla\psi|^2)^{3/2}nH}{(\varrho^{-2}+|\nabla\psi|^2)^{3/2}}
	\end{align*}
and hence we only need to find $R_3=R_3(C,C_0,\delta)\ge R_2$ so that
	\begin{align}\label{psi_barr_claim}
&	(\varrho^{-2}+|\nabla\psi|^2)^{3/2}\mathcal{Q}[\psi]\\
&=	(\varrho^{-2}+|\nabla\psi|^2)\Delta_{-\log\varrho}\psi - \frac{1}{2} \ang{\nabla(\varrho^{-2}+|\nabla\psi|^2),\nabla\psi} -  (\varrho^{-2}+|\nabla\psi|^2)^{3/2}nH < 0\nonumber
	\end{align}
holds in the set $3\Omega\setminus \bar B(o,R_3)$.

The function $\psi$ is $C^{\infty}$-smooth and, in $M\setminus\{o\}$, we have
	\[
	\nabla \psi = A(-R_3^{\delta}\delta r^{-\delta-1}\nabla r + \nabla h).
	\]
By Lemma~\ref{arvio_lause}, $|\nabla h| \le c_1/ f_a(r)\le \delta r^{-\delta-1}$ when
$r$ is large enough and $0<\delta<\min\{\delta_1,\phi-1\}$; see \cite[(3.30)]{HoVa}. Hence, for any fixed $\ve>0$, 
we have
	\begin{align*}
		|\nabla \psi|^2 &= (AR_3^\delta \delta)^2 r^{-2\delta-2} + A^2|\nabla h|^2 - 2A^2R_3^\delta \delta 
      r^{-\delta-1} \ang{\nabla r,\nabla h} \\
&      \le A^2\delta^2\bigl(R_3^{2\delta}+2R_3^\delta +1\bigr) r^{-2\delta-2}\\
&\le 
      (1+\varepsilon ) (AR_3^\delta\delta)^2  r^{-2\delta-2}
	\end{align*}
and
\[
|\nabla \psi|^2  \ge A^2\delta^2\bigl( R_3^{2\delta}-2R_3^\delta\bigr)r^{-2\delta-2}
\ge (1-\varepsilon)  (AR_3^\delta\delta)^2  r^{-2\delta-2}
	\]
in $3\Omega\setminus \bar B(o,R_3)$ for $R_3$ large enough.

Next we fix $\ve>0$ so that 
\begin{equation}\label{eps-bound}
\ve < 1-\frac{\delta+1}{(n-1)(1-\delta)\phi},
\end{equation}
which is possible since $\delta<\delta_1$. To simplify the notation below, we denote $\tilde{\ve}=\ve\,\sgn(\p_r\varrho)$.
In order to estimate the first term in the right-hand side of \eqref{psi_barr_claim}, we first observe that
\begin{equation}\label{nega-term}
-(n-1)\frac{rf_a'(r)}{f_a(r)} 
		-  \frac{r\p_r \varrho}{\varrho}  + \frac{\delta+1}{1-\varepsilon}<0
\end{equation}
for $r\ge R_3$ by \eqref{rho_sign_assum} and \eqref{eps-bound}; see \cite[(3.25)]{HoVa}. Then we can estimate the weighted Laplacian of $\psi$ as
	\begin{align*}
	\Delta_{-\log\varrho} \psi &= AR_3^{\delta} \Delta_{-\log\varrho} r^{-\delta} + A\Delta_{-\log\varrho} h \\
	&= AR_3^{\delta} \left(\Delta r^{-\delta} + \frac{1}{\varrho} \ang{\nabla\varrho,\nabla(r^{-\delta})}\right)
		+ A\left(\Delta h + \frac{1}{\varrho} \ang{\nabla \varrho,\nabla h} \right) \\
	&= AR_3^{\delta}\left(-\delta r^{-\delta-1} \Delta r 
		-\delta r^{-\delta-1} \frac{1}{\varrho} \ang{\nabla\varrho,\nabla r} + \delta(\delta+1)r^{-\delta-2} \right) \\ 
	 	&\qquad +A\left(\Delta h + \frac{1}{\varrho} \ang{\nabla \varrho,\nabla h} \right) \\
	&\le AR_3^{\delta}\delta \left(-  (n-1)\frac{rf_a'(r)}{f_a(r)} 
		-  \frac{r\p_r \varrho}{\varrho}  + \delta+1 \right)r^{-\delta-2} \\
		&\qquad + A\left(nc_2 r^{-C_4-1}\frac{f_a'(r)}{f_a(r)} + \frac{c_2|\nabla \varrho|}{\varrho f_a(r)}  \right) \\
	&\le  AR_3^{\delta}\delta \left(\frac{-(1-\ve)(n-1)rf_a'(r)}{f_a(r)} 
		-  \frac{(1-\tilde{\ve})r\p_r \varrho}{\varrho}  + \delta+1 \right)r^{-\delta-2}<0
	\end{align*}
for  $r\ge R_3$. In the last step we used \eqref{rho_sign_assum}, 
\eqref{rho_rad_assum}, and the fact that $C_4>\delta$.
Hence
\begin{align}\label{psi_barr_1term}
	&(\varrho^{-2}+|\nabla\psi|^2) \Delta_{-\log\varrho}\psi \nonumber\\
	 &\ \le -\bigl(\varrho^{-2}+ (1-\ve)(A R_3^\delta \delta)^2 r^{-2\delta-2}\bigr) AR_3^\delta \delta
\biggl( \frac{(1-\ve)(n-1)rf_a'(r)}{f_a(r)} \\
 &\qquad +  \frac{(1-\tilde{\ve})r\p_r \varrho}{\varrho} -1-\delta \biggr)r^{-\delta-2}.\nonumber
\end{align}

To estimate the second term of \eqref{psi_barr_claim} we split it into two parts as
	\[
	-\frac{1}{2}\ang{\nabla(\varrho^{-2}+ |\nabla\psi|^2),\nabla \psi} = 
	-\frac{1}{2}\ang{\nabla(\varrho^{-2}),\nabla\psi} 
	-\frac{1}{2}\ang{\nabla|\nabla\psi|^2,\nabla\psi}.
	\] 
For the first term, by \eqref{rho_rad_assum} and Lemma \ref{perusta}, we have
	\begin{align}\label{psi_bar_2_1term}
	-\frac{1}{2}\ang{\nabla(\varrho^{-2}),\nabla \psi} &= \ang{\frac{\nabla\varrho}{\varrho^3},
	\nabla \psi}  = \ang{\frac{\nabla\varrho}{\varrho^3}, -AR_3^\delta \delta r^{-\delta-1}\nabla r} 
	+ \ang{\frac{\nabla\varrho}{\varrho^3},A\nabla h} \nonumber \\
	&\le -AR_3^\delta \delta r^{-\delta-1} \frac{\p_r\varrho}{\varrho^3} 
	+ c_2 A \frac{|\nabla\varrho |}{\varrho^3f_a(r)}  \\
	&\le -(1-\tilde{\ve})AR_3^\delta \delta  r^{-\delta-1}\frac{\p_r\varrho}{\varrho^3}.\nonumber
	\end{align}
To estimate the second term we note that
	\begin{align*}
	\nabla|\nabla\psi|^2 &= A^2\nabla\ang{R_3^{\delta}\nabla(r^{-\delta})
	+\nabla h,R_3^{\delta}\nabla(r^{-\delta})+\nabla h}\\
	&=(AR_3^{\delta})^2\nabla\ang{\nabla(r^{-\delta}),\nabla(r^{-\delta})} 
	+2A^2 R_3^{\delta}\nabla\ang{\nabla(r^{-\delta}),\nabla h}
	+A^2 \nabla\ang{\nabla h,\nabla h}
	\end{align*}
and hence,
by a straightforward computation using the estimates of Lemma \ref{perusta}, we get
	\begin{align}\label{psi_bar_2_2term}
	-\frac{1}{2}\ang{\nabla|\nabla\psi|^2,\nabla\psi} &= 
	-\frac{1}{2}(AR_3^{\delta})^2\ang{\nabla\ang{\nabla(r^{-\delta}),\nabla(r^{-\delta})},\nabla\psi}\nonumber \\
	&\quad -  A^2 R_3^{\delta}\ang{\nabla\ang{\nabla(r^{-\delta}),\nabla h},\nabla\psi} 
	-\frac{1}{2} A^2 \ang{\nabla\ang{\nabla h,\nabla h},\nabla\psi} \nonumber \\
	&\le (AR_3^\delta \delta)^2 (\delta+1)r^{-2\delta-3} \ang{\nabla r,\nabla\psi} 
	+ A^2 R_3^\delta c_2 r^{-C_4 -2} \frac{f_a'(r)}{f_a(r)}|\nabla\psi| \nonumber \\
	&\quad + \frac 12 A^2 c_2r^{-C_4-2}\frac{f_a'(r)}{f_a(r)}|\nabla\psi|  \\
	&\le (AR_3^\delta \delta)^2 (\delta+1)r^{-2\delta-3} \ang{\nabla r,-AR_3^\delta \delta r^{-\delta-1}\nabla r 
	+ A\nabla h} \nonumber \\
	&\quad + c r^{-C_4-\delta-3} \frac{f_a'(r)}{f_a(r)} \nonumber \\
	&\le -c r^{-3\delta-4} + c r^{-2\delta-3}\frac{1}{f_a(r)} + c r^{-C_4-\delta-3} \frac{f_a'(r)}{f_a(r)}\nonumber\\
	&\le -c r^{-3\delta-4}   + c r^{-C_4-\delta-3} \frac{f_a'(r)}{f_a(r)},\nonumber
	\end{align}
where in the last step we have absorbed the term $c r^{-2\delta-3}\frac{1}{f_a(r)}$ into the first by using the fact that
$f_a(r)\ge cr^\phi$ and the choice of $\delta<\phi-1$.
Putting together \eqref{psi_bar_2_1term} and \eqref{psi_bar_2_2term} we get
	\begin{equation*}
		-\frac{1}{2}\ang{\nabla(\varrho^{-2}+ |\nabla\psi|^2),\nabla \psi} 
	\le
	-(1-\tilde{\ve})AR_3^\delta\delta r^{-\delta-1} \frac{\p_r\varrho}{\varrho^3} 
	-c r^{-3\delta-4} + c r^{-C_4-\delta-3} \frac{f_a'(r)}{f_a(r)},
	\end{equation*}
and combining this with \eqref{psi_barr_1term} yields
\begin{align}\label{psi_barr_2term}
&(\varrho^{-2}+|\nabla\psi|^2)\Delta_{-\log\varrho}\psi
-\frac{1}{2}\ang{\nabla(\varrho^{-2}+ |\nabla\psi|^2),\nabla \psi} \nonumber\\
	 & \quad\le 
-\frac{AR_3^\delta\delta}{\varrho^2}\left(\frac{(1-\ve)(n-1)rf_a'(r)}{f_a(r)}+\frac{2(1-\tilde{\ve})r\p_r\varrho}{\varrho}-\delta-1\right)r^{-\delta-2} \\
&\quad\ -(1-\ve)(AR_3^\delta\delta)^3\left(\frac{(1-\ve)(n-1)rf_a'(r)}{f_a(r)}+\frac{(1-\tilde{\ve})r\p_r\varrho}{\varrho}-1-\delta+c\right)r^{-3\delta-4},
\nonumber 
	 \end{align}
where we have absorbed the positive term $c r^{-C_4-\delta-3} f_a'(r)/f_a(r)$ by using the assumption $\delta<C_4/2$.
Finally, using the assumption \eqref{asym_meancurv_assum} we can estimate the term involving the mean curvature as 
	\begin{align}\label{mean-term-est}
	-&(\varrho^{-2}+|\nabla\psi|^2)^{3/2}nH\nonumber \\
	&\le (1+\ve)^{3/2} ( \varrho^{-2}+ (AR_3^\delta \delta)^2 r^{-2\delta-2})^{3/2} n|H|\\
	&\le 
	\frac{c}{\varrho^2}\left(\frac{(n-1)rf_a'(r)}{f_a(r)}+\frac{r\p_r\varrho}{\varrho}-1\right)r^{-\delta_1-2} \nonumber\\
&\quad	+c\left(\frac{(n-1)rf_a'(r)}{f_a(r)}+\frac{r\p_r\varrho}{\varrho}-1\right)r^{-2\delta-\delta_1-4}.\nonumber
	\end{align}
Combining  \eqref{psi_barr_2term} and \eqref{mean-term-est} and noting that $\delta_1>\delta$ we obtain \eqref{psi_barr_claim}
and the claim follows.
\end{proof}

\begin{rem}
In the case of the hyperbolic (ambient) space 
$\HH^{n+1}=\HH^n\times_{\cosh r}\R$ we have $\varrho=\varrho_+(r)=\cosh r$ and $f_a(r)=\sinh r$ on $\HH^n$ for any reference point $o\in \HH^n$.  Hence \eqref{rho_sign_assum} and \eqref{rho_rad_assum} hold trivially. Moreover, we may choose $\phi>1$ as large as we wish by increasing $R_3$ and therefore \eqref{eps-bound} and \eqref{nega-term} hold even with $\delta=\delta_1$. Finally,  
\begin{equation*}
-(\varrho^{-2}+|\nabla\psi|^2)^{3/2}nH
\le (1+\ve)(AR_3^\delta \delta)^3 r^{-3\delta-3}) n|H|
\end{equation*}
for $r$ large enough, and consequently we may assume $\delta=\delta_1$ in \eqref{asym_meancurv_assum} thus reducing it to an asymptotically sharp assumption.

Similarly, if the sectional curvatures of $M$ have estimates 
\[
-r(x)^{-2-\varepsilon}e^{2kr(x)}\le K(P_x) \le -k^2
\]
for $r(x)\ge R_0$ as in Example~\ref{ex2} and if the warping function $\varrho$ satisfies \eqref{rho_sign_assum},   \eqref{rho_rad_assum}, and 
\[
\varrho(x)\ge c r(x)^2
\]
for $r(x)\ge R_0$, we may take $\delta=\delta_1$ in \eqref{asym_meancurv_assum}.
\end{rem}

\section{Solving the asymptotic Dirichlet problem}\label{ADP_sec}
In this section we solve the asymptotic Dirichlet problem \eqref{ADP}
on a Cartan-Hadamard manifold $M$ with given boundary data $\varphi\in C(\partial_\infty M)$.
If the ambient manifold $N=M\times_\varrho\R$ is a Cartan-Hadamard manifold, too, we will interpret the graph 
$S=\{(x,u(x))\colon x\in M\}$ of the solution $u$ as a Killing graph with prescribed mean curvature $H$
and continuous boundary values at infinity. We recall from 
\cite[7.7]{BO} that $N$ is a Cartan-Hadamard manifold if and only if the warping function $\varrho$ is convex.
In that case we may consider $\partial_\infty M$ as a subset of $\partial_\infty N$ in the sense that a representative $\gamma$ of 
a boundary point $x_0\in\partial_\infty M$ is also a representative of a point $\tilde{x}_0\in \partial_\infty N$ since
$M$ is a totally geodesic submanifold of $N$. Given $\varphi\in C(\partial_\infty M)$ we define its Killing graph on 
$\partial_\infty N$ as follows. For $x\in\partial_\infty M$, take the (totally geodesic) leaf 
\[
M_{\varphi(x)}=\Psi(M,\varphi(x))=
\{(y,\varphi(x))\colon y\in M\}\subset M\times \R,
\]
where $\Psi$ is the flow generated by $X$. Let $\gamma^{x}$ be any geodesic on $M$ representing $x$. Then 
$\tilde{\gamma}^{x}\colon t\mapsto\Psi(\gamma^x(t),\varphi(x))$ is a geodesic on $M_{\varphi(x)}$ and also on $N$ since 
$\Psi(\cdot,\varphi(x))$ is an isometry. Hence $\tilde{\gamma}^x$ defines a point in $\partial_\infty N$ which we, 
by abusing the notation, denote by $(x,\varphi(x))$. Using this notation, we call the set 
\[
\Gamma=\{(x,\varphi(x))\colon x\in\partial_\infty M\}\subset\partial_\infty N
 \]
 the Killing graph of $\varphi$. Note that, in general, $\partial_\infty N$ has no canonical smooth structure.
\begin{lem}\label{boundSisGamma}
  Let $u$ be the solution to \eqref{ADP} with boundary data $\varphi$ and let $S$ be the graph of $u$.
 If $\partial_\infty S=\bar{S}\setminus S$, where $\bar{S}$ is the closure of $S$ in the cone topology $\bar{N}$,
 we have $\partial_\infty S=\Gamma.$ 
  \end{lem}
\begin{proof}
Suppose first that $x\in\partial_\infty S$ and let $(x_i,u(x_i))$ be a sequence in $S$ converging to $x$ in the cone topology of 
$\bar{N}$.
Since $\bar{M}$ is compact, there exist $x_0\in\partial_\infty M$ and  a subsequence $(x_{i_j},u(x_{i_j}))$ such that $x_{i_j}\to x_0\in\partial_\infty M$
in the cone topology of $\bar{M}$. Hence $u(x_{i_j})\to \varphi(x_0)$, and consequently 
$(x_{i_j},u(x_{i_j}))\to (x_0,\varphi(x_0))$ in the product topology of $\bar{M}\times\R$. On the other hand, 
$\Psi(x_{i_j},\varphi(x_0))\to (x_0,\varphi(x_0))$ in the cone topology of $M_{\varphi(x_0)}$. We need to verify that 
$\Psi(x_{i_j},u(x_{i_j}))\to (x_0,\varphi(x_0))$ in the cone topology of $\bar{N}$ which then implies that $x=(x_0,\varphi(x_0))\in\Gamma$. Towards this end, let $V$ be an arbitrary cone neighborhood in $\bar{N}$ of $(x_0,\varphi(x_0))$
and let $\sigma$ be a geodesic ray emanating from $(o,\varphi(x_0))$ representing $(x_0,\varphi(x_0))$. It is a geodesic ray both in $N$ and in $M_{\varphi(x_0)}$. Let 
$T(\dot{\sigma}_0,2\alpha,r)\subset V$ be a truncated cone in $\bar{N}$ and 
$T:=T^M(\dot{\sigma}_0,\alpha,2r)$ a truncated cone in $\bar{M}_{\varphi(x_0)}$. Then $\Psi(T,(\varphi(x_0)-\delta,\varphi(x_0)+\delta))\subset V$ for sufficiently small $\delta>0$. It follows that 
$\Psi(x_{i_j},u(x_{i_j}))\in V$ for all $i_j$ large enough, and therefore
$x=(x_0,\varphi(x_0))\in\Gamma$.

Conversely, if $(x_0,\varphi(x_0))\in\Gamma$, let $x_i\in M$ be a sequence such that $x_i\to x_0$ in the cone topology of $\bar{M}$.
Then $\Psi(x_i,u(x_i))\in S$ and $(x_i,u(x_i))\to (x_0,\varphi(x_0))$ in the product topology of $\bar{M}\times\R$. We need to show that $\Psi(x_i,u(x_i))\to (x_0,\varphi(x_0))\in\Gamma$ in the cone topology of $\bar{N}$. To prove this, fix $o=\Psi(x,\varphi(x_0))\in M_{\varphi(x_0)}$ and let $\sigma$ be a geodesic ray in $N$ 
(and in $M_{\varphi(x_0)}$) representing $(x_0,\varphi(x_0))$. Let 
$V=T(\dot{\sigma}_0,2\alpha,r)$ be an arbitrary truncated cone
neighborhood in $\bar{N}$ of $(x_0,\varphi(x_0))$. Furthermore, let $\delta>0$ be so small that $U:=\Psi(\tilde{V},(\varphi(x_0)-\delta,\varphi(x_0)+\delta))\subset V$, where $\tilde{V}=T(\dot{\sigma}_0,\alpha,2r)$ is a truncated cone neighborhood in $M_{\varphi(x_0)}$ of $(x_0,\varphi(x_0))$. Since $x_i\to x_0$ and $u(x_i)\to \varphi(x_0)$,
we obtain $\Psi(x_i,u(x_i))\in U$ for all sufficiently large $i$.
Hence $\Psi(x_i,u(x_i))\to (x_0,\varphi(x_0))\in\Gamma$ in the cone topology of $\bar{N}$.
\end{proof}

We formulate our global existence results in the following two theorems depending on the assumption on the prescribed mean curvature function $H$.
\begin{thm}\label{entire-exist1}
Let $M$ be a Cartan-Hadamard manifold satisfying the curvature
assumptions \eqref{curv-bound-gen} and \eqref{A1}--\eqref{A7} in Section \ref{sec_bar_infty}. Furthermore, assume that the prescribed mean curvature function 
$H\colon M \to \R$ satisfies the assumptions \eqref{height_mean_assum2} and 
\eqref{asym_meancurv_assum} with a convex warping function $\varrho$ satisfying \eqref{rho_height_assum}, \eqref{rho_height_assum1}, \eqref{rho_sign_assum}, and \eqref{rho_rad_assum}.
Then there exists a unique solution $u\colon M\to \R$
to the Dirichlet problem
	\begin{equation}\label{ADP}
		\begin{cases}
	\dv_{-\log \varrho} \dfrac{\nabla u}{\sqrt{\varrho^{-2} + |\nabla u|^2}} = nH(x) \quad \text{in } M \\
	u|\pinf M = \varphi
	\end{cases}
		\end{equation}
for any continuous function $\varphi\colon \pinf M\to \R$.
\end{thm}

\begin{thm}\label{entire-exist2}
Let $M$ be a Cartan-Hadamard manifold satisfying the curvature
assumptions \eqref{curv-bound-gen} and \eqref{A1}--\eqref{A7} in Section \ref{sec_bar_infty}. Furthermore, assume that the prescribed mean curvature function 
$H\colon M \to \R$ satisfies the assumptions \eqref{HVbound} and 
\eqref{asym_meancurv_assum} with a convex warping function $\varrho$ satisfying \eqref{rho+assumption2}, \eqref{rho_sign_assum}, and \eqref{rho_rad_assum}.
Then there exists a unique solution $u\colon M\to \R$
to the Dirichlet problem \eqref{ADP}
for any continuous function $\varphi\colon \pinf M\to \R$.
\end{thm}

\begin{proof} The proofs of Theorems \ref{entire-exist1} and \ref{entire-exist2} are similar. The only difference is to use the global 
barrier $u_+$ in Lemma~\ref{unif_height_estim} for \ref{entire-exist1} relative to $V$ in Lemma~\ref{unif-heigh2} for \ref{entire-exist2}. 

Extend the boundary data function $\varphi\in C(\pinf M)$ to a function $\varphi\in C(\bar M)$ and let
$B_k=B(o,k),\, k\in\N$ be an exhaustion of $M$. Then by Corollary~\ref{loc-exist-cor} there exist solutions
$u_k\in C^{2,\alpha}(B_k)\cap C(\bar B_k)$ to the Dirichlet problem 
	\[
	\begin{cases}
	\dv_{-\log \varrho} \dfrac{\nabla u_k}{\sqrt{\varrho^{-2} + |\nabla u_k|^2}} = nH(x) \quad \text{in } B_k \\
	u_k|\p B_k = \varphi.
	\end{cases}
	\]
By Lemma \ref{unif_height_estim},  we see that the sequence $(u_k)$ is uniformly bounded. 
Applying the gradient estimates in compact domains and then the diagonal argument, we 
obtain a subsequence converging
locally uniformly with respect to $C^2$-norm to a solution $u$. Next we show that $u$ extends
continuously to the boundary $\pinf M$ with $u|\pinf M = \varphi$.

Let $x_0\in\pinf M$ and $\ve>0$ be fixed. By the continuity of the function $\varphi$ we find
a constant $L\in(8/\pi,\infty)$ so that
	\[
	|\varphi(y)-\varphi(x_0)|<\ve/2
	\] 
whenever $y\in C(v_0,4/L) \cap\pinf M$, where $v_0 = \dot\gamma_0^{o,x_0}$ is the initial direction of
the geodesic ray representing $x_0$. 
Taking \eqref{u+_infty} into account, we can choose $R_3$ in Lemma \ref{psi_barrier_lemma} so big
that $u_+(r) \le ||\varphi||_\infty + \ve/2$ when $r\ge R_3$.

We will show that
	\begin{equation}\label{u_pinched}
	w^-(x) \coloneqq -\psi(x) + \varphi(x_0) -\ve \le u(x) \le w^+(x) \coloneqq \psi(x) + \varphi(x_0)+\ve
	\end{equation}
in the set $U\coloneqq 3\Omega\setminus \bar B(o,R_3)$. Here $\psi = A(R_3^\delta r^{-\delta}+h)$
is the supersolution from the Lemma \ref{psi_barrier_lemma} and $A = 2||\varphi||_\infty$.

Again, by the continuity of the function $\varphi$ in $\bar M$, we can choose $k_0$ such that 
$\p B_k \cap U \neq \emptyset$ and 
	\begin{equation}\label{phi_cont_cond}
	|\varphi(x)-\varphi(x_0)|<\ve/2
	\end{equation}
for every $x\in \p B_k\cap U$ when $k\ge k_0$. We denote $V_k = B_k\cap U$ for $k\ge k_0$ and
note that 
	\[
	\p V_k = (B_k \cap \bar U) \cup (\p U \cap \bar B_k).
	\]
We prove \eqref{u_pinched} by showing that
	\begin{equation}\label{u_k_pinched}
	w^-\le u_k \le w^+
	\end{equation}
holds in $V_k$ for every $k\ge k_0$.

Let $k\ge k_0$ and $x \in \p B_k \cap \bar U$. Since $u_k| \p B_k = \varphi |\p B_k$, \eqref{phi_cont_cond}
implies
	\[
	w^-(x) \le \varphi(x_0) - \ve/2 \le \varphi(x) = u_k(x) \le \varphi(x_0) + \ve/2 \le w^+(x).
	\]
By Lemma \ref{arvio_lause} 
	\[
	h|M\setminus \big( 2\Omega \cup B(o,R_1) \big) = 1
	\]
and since $R_3^{\delta} r^{-\delta} =1$ on $\p B(o,R_3)$ we have
	\[
	\psi \ge A= 2||\varphi||_\infty
	\]
on $\p U\cap B_k$. Since $u_+$ from Lemma \ref{unif_height_estim} is global supersolution with 
$u_+ \ge ||\varphi||_\infty$ on $\p B_k$, the comparison principle gives $u_k|B_k \le u_+|B_k$ and
by the choice of $R_3$, we have
	\[
	u_k \le ||\varphi||_\infty + \ve/2
	\]
in the set $B_k\setminus B(o,R_3)$.

Putting all together, it follows that
	\[
	w^+ = \psi + \varphi(x_0) +\ve \ge 2 ||\varphi||_\infty + \varphi(x_0) + 
	\ve \ge||\varphi||_\infty + \ve \ge u_k
	\]
on $\p U\cap \bar B_k$. Similarly we have $u_k\ge w^-$ on $\p U\cap \bar B_k$ and therefore
$w^-\le u_k \le w^+$ on $\p V_k$. By Lemma \ref{psi_barrier_lemma} $\psi$ is a supersolution
in $U$ and hence the comparison principle yields $u_k\le w^+$ in $U$. On the other hand, $-\psi$
is a subsolution in $U$, so $u_k\ge w^-$ in $U$, and \eqref{u_k_pinched} follows. This is true for every $k\ge k_0$
so we have \eqref{u_pinched}. Since $\lim_{x\to x_0}\psi(x) = 0$, we have
	\[
	\limsup_{x\to x_0} |u(x) - \varphi(x_0)| \le \ve. 
	\]
The point $x_0\in\pinf M$ and constant $\ve>0$ were arbitrary so this shows that $u$ extends continuously
to $C(\bar M)$ and $u|\pinf M = \varphi$. Finally, the uniqueness follows from the comparison principle.

\end{proof}

\section{Non-existence result}

In the following, we state a non-existence result for the prescribed weighted mean curvature graph equation by adapting the approach of Pigola, Rigoli and Setti in \cite{PigolaRigoliSetti}.
We denote by $A(r)$ the area of the geodesic sphere $\partial B(o,r)$ centred at a fixed point $o\in M$.

\begin{prop}
Let $p\colon[0,\infty) \to [0,\infty)$ be a continuous  function 
such that for some $\bar{R}>0$ and for all $r\ge \bar{R}$ at least one of the following conditions is satisfied:
	\begin{equation}\label{volume-assumption}
	\dfrac{\exp\left(D\left(\int_0^r\sqrt{p(s)}{\rm d}s\right)^2\right)}
{\varrho_0(r)^2 A(r)} \notin L^1(+\infty)
	\end{equation}
for some constant $D>0$ and  a smooth function $\varrho_0$, so that $\varrho(x)\le \varrho_0(r(x))$, or 
	\begin{equation}\label{volume-assumption2}
\dfrac{\left(\int_r^{3r/2}\sqrt{p(s)}{\rm d}s\right)^2}{r \log \big(\varrho_0(2r)^2 \vol(B(o,2r))\big)} \ge h (r) \notin L^1(+\infty)
	\end{equation}
with some continuous and monotonically non-increasing 
$h\colon [\bar R,\infty) \to (0,\infty)$.
Let $u,v \in C^2(M)$ satisfy
	\begin{align}\label{nonex_prop}
\dv_{-\log\varrho} \dfrac{\nabla u}{\sqrt{\varrho^{-2}+ |\nabla u|^2}} &- \dv_{-\log\varrho} \dfrac{\nabla v}{\sqrt{\varrho^{-2}+ |\nabla v|^2}}
	= q(x)\nonumber\\
	& \ge  p\big(r(x)\big) \varrho_0\big(r(x)\big) \ge 0, \\
	\noalign{and}
	\sup_M (u-v) &< +\infty.\nonumber
	\end{align}
Then, if $q\not\equiv 0$, there are no solutions to \eqref{nonex_prop}.
\end{prop}

\begin{proof}
The proof is very similar to that in \cite{PigolaRigoliSetti}, the only differences being our use of the divergence operator with respect to the weighted volume form $\varrho {\rm d}M$ and a suitable form of the Mikljukov-Hwang-Collin-Krust inequality which in our setting reads as follows
	\begin{align*}
	&\ang{\frac{\nabla u}{\sqrt{\varrho^{-2}+ |\nabla u|^2}} - \frac{\nabla v}{\sqrt{\varrho^{-2}+ |\nabla v|^2}},
	\nabla u - \nabla v} 
\\	
&	\ge \frac{1}{2} \left( \sqrt{\varrho^{-2} + |\nabla u|^2} + \sqrt{\varrho^{-2} + 
	|\nabla v|^2} \right)   
\left| \frac{\nabla u}{\sqrt{\varrho^{-2}+ |\nabla u|^2}} - \frac{\nabla v}{\sqrt{\varrho^{-2}+ |\nabla v|^2}} \right|^2
\\
& \ge \varrho^{-1}
\left| \frac{\nabla u}{\sqrt{\varrho^{-2}+ |\nabla u|^2}} - \frac{\nabla v}{\sqrt{\varrho^{-2}+ |\nabla v|^2}} \right|^2.
	\end{align*}
Together these result in the extra factors of $\varrho_0$ in 
\eqref{volume-assumption}, \eqref{volume-assumption2}, and on the right hand side of \eqref{nonex_prop}. Taking into account these differences the proof in \cite{PigolaRigoliSetti} applies almost verbatim.
\end{proof}

As direct corollaries of the previous theorem, we have
\begin{cor}
Let $u$ be a bounded solution to 
\[
\dv_{-\log \varrho} \dfrac{\nabla u}{\sqrt{\varrho^{-2} + |\nabla u|^2}} = nH(x) \quad \text{in } M, 
\]
with $H\ge 0$. 
\begin{enumerate}
\item[(i)]
Suppose that 
$\varrho(x)\le\varrho_0\big(r(x)\big)\le r(x)^{\beta_1},\ \beta_1>0$, and that $A(r)\le r^{\beta_2},\ \beta_2>0$, for large values of $r=r(x)$. Then
\[
\liminf_{r(x)\to\infty}H(x)\cdot\frac{r(x)^2\log r(x)}{\varrho_0\big(r(x)\big)}
=0.
\]
\item[(ii)]
Suppose that 
$\varrho(x)\le\varrho_0\big(r(x)\big)\le e^{\beta_1 r(x)},\ \beta_1>0$, and that $A(r)\le e^{\beta_2 r},\ \beta_2>0$, for large values of $r=r(x)$. Then
\[
\liminf_{r(x)\to\infty}H(x)\cdot\frac{r(x)\log r(x)}{\varrho_0\big(r(x)\big)}
=0.
\]
\item[(iii)]
Suppose that 
$\varrho(x)\le\varrho_0\big(r(x)\big)\le e^{\beta_1 r(x)^2},\ \beta_1>0$, and that $A(r)\le e^{\beta_2 r^2},\ \beta_2>0$, for large values of $r=r(x)$. Then
\[
\liminf_{r(x)\to\infty}H(x)\cdot\frac{\log r(x)}{\varrho_0\big(r(x)\big)}
=0.
\]
\end{enumerate}
\end{cor}
\begin{proof}
By choosing $p(s)=(s^2 \log s)^{-1}$ in (i), we see that \eqref{volume-assumption} holds, and therefore the claim follows. Similarly, choosing $p(s)=(s\log s)^{-1}$ in (ii) or $p(s)=(\log s)^{-1}$ in (iii), the condition \eqref{volume-assumption2} holds and the claim follows.
\end{proof}

\end{document}